\documentclass[preprint,review,1p]{elsarticle}
 \usepackage{amssymb}
 \usepackage{amsmath}
 \usepackage{amsthm}
\newtheorem{theorem}{Theorem}[section]

\newtheorem{lemma}[theorem]{Lemma}
\newtheorem{remark}[theorem]{Remark}
\newtheorem{definition}[theorem]{Definition}

\journal{Journal of Mathematical Analysis and Applications}

\begin{document}

 \begin{frontmatter}

 \title{Existence and Stability of Traveling Waves for a Class of Nonlocal Nonlinear Equations}

\author{H. A. Erbay$^{1}$\corref{cor1}}
    \ead{husnuata.erbay@ozyegin.edu.tr}
\author{S. Erbay$^1$}
    \ead{saadet.erbay@ozyegin.edu.tr}
\author{A. Erkip$^2$}
    \ead{albert@sabanciuniv.edu}
\cortext[cor1]{Corresponding author. Tel: +90 216 564 9489 Fax: +90 216 564 9057}

\address{$^1$ Department of Natural and Mathematical Sciences, Faculty of Engineering, Ozyegin University,  Cekmekoy 34794, Istanbul, Turkey}

\address{$^2$ Faculty of Engineering and Natural Sciences, Sabanci University,  Tuzla 34956,  Istanbul,    Turkey}

  \begin{abstract}
 In this article we are concerned with the existence and orbital stability of traveling wave solutions of a general class of nonlocal wave equations: $~u_{tt}-Lu_{xx}=B(\pm |u|^{p-1}u)_{xx}$,  $~p>1$. The main characteristic of this class of equations is the existence of two sources of dispersion, characterized by two coercive pseudo-differential operators $L$ and $B$. Members of the class arise as mathematical models for the propagation of dispersive waves in a wide variety of situations.  For instance, all Boussinesq-type equations and the so-called double-dispersion equation are members of the class.  We first establish the existence of traveling wave solutions to the nonlocal wave equations considered. We then obtain results on the orbital stability or instability of  traveling waves. For  the case  $L=I$, corresponding to a class of Klein-Gordon-type equations,  we give an almost complete characterization of  the values of the wave velocity for which the traveling waves are orbitally stable or unstable by blow-up.
 \end{abstract}

\begin{keyword}
    Solitary waves  \sep Orbital stability \sep  Boussinesq equation \sep Double dispersion equation  \sep Concentration-compactness \sep Instability by blow-up \sep Klein-Gordon equation.
    \MSC  74H20 \sep 74J30 \sep 74B20
\end{keyword}
\end{frontmatter}

\setcounter{equation}{0}
\section{Introduction}
\noindent
The present paper is concerned with the existence and stability of traveling wave solutions $u(x,t)=\phi _{c}(x-ct)$ of a general class of nonlocal nonlinear equations of the form
\begin{equation}
    u_{tt}-Lu_{xx}=B(g(u))_{xx}, ~~~x\in \mathbb{R}, ~~~t>0, \label{nonlocal}
\end{equation}%
where $c\in \mathbb{R}$ is the wave velocity, $u=u(x,t)$ is a real-valued function,  $g(u)=\pm |u|^{p-1}u$ with $p>1$, and  $L$ and $B$ are linear  pseudo-differential operators with smooth symbols $l(\xi )$ and $b(\xi )$, respectively. The orders of $L$ and $B$ will be denoted by $\rho$ and $-r$, respectively. Here, and throughout this paper, we assume that {\it (i)} $r\geq 0$,   {\it (ii)} for all $k$  the symbols $l(\xi )$ and $b(\xi )$ satisfy the decay properties
\begin{equation}
    \frac{d^{k}}{d\xi ^{k}}l(\xi )=\mathcal{O}(|\xi |^{\rho -k}),~~~~~\frac{d^{k}}{d\xi ^{k}}b(\xi )
            =\mathcal{O}(|\xi |^{-r-k})~~~~\mbox{as}~~|\xi | \rightarrow \infty,   \label{derivestimate}
\end{equation}
and  {\it (iii)} the pseudo-differential operators $L$ and $B$  are coercive elliptic operators; namely there exist positive constants $~c_{1}, c_{2}, c_{3}$ and $c_{4}~$ such that
\begin{eqnarray}
 &&   c_{1}^{2}(1+\xi ^{2})^{\rho /2}\leq l(\xi )\leq c_{2}^{2}(1+\xi ^{2})^{\rho/2}, \label{bn-a}\\
 &&  c_{3}^{2}(1+\xi ^{2})^{-r/2} \leq b(\xi )\leq c_{4}^{2}(1+\xi^{2})^{-r/2},  \label{bn-b}
\end{eqnarray}%
for all $\xi \in \mathbb{R}$. Throughout the study we assume that the above constants $c_{i}$ are chosen as the best constants. The aim of the present study is twofold: first to show the existence of  traveling wave solutions $u(x,t)= \phi_c(x-ct)$ of (\ref{nonlocal}) for the above-defined class of pseudo-differential operators $L$ and $B$, and then to investigate the orbital stability and instability of those traveling wave solutions.

Equation (\ref{nonlocal}) was first proposed in \cite{babaoglu} as a general equation governing the propagation of doubly dispersive nonlinear waves.
The local existence, global existence and blow-up results for solutions of the Cauchy problem of (\ref{nonlocal}) with initial data in suitable Sobolev spaces were provided in \cite{babaoglu}. In a recent study \cite{erbay2}, thresholds for global existence  versus blow-up were established for (\ref{nonlocal}) with power-type nonlinearities.

To illustrate the double nature of dispersion we rewrite (\ref{nonlocal}) in the form $B^{-1}u_{tt}-LB^{-1}u_{xx}=(g(u))_{xx}$, where the first and second terms on the left-hand side represent two sources of dispersive effect.  Clearly, for suitable choices of $L$ and $B$, (\ref{nonlocal}) will reduce to the well-known Boussinesq-type equations, including the Boussinesq equation \cite{bous}, the improved Boussinesq equation \cite{ost} and the double dispersion equation  \cite{samsonov, porubov} (see Section 3 of the present study and \cite{babaoglu} for further details). In recent years, higher-order effects leading to higher-order spatial derivatives become more important for   modeling dispersive wave propagation in micro-structured solids \cite{berezovski}. We also observe that  nonlocal operators with non-polynomial symbols  appear in modeling studies of nanotechnology applications \cite{narendar}. It is interesting to note that both of these models give rise to wave equations of the form (\ref{nonlocal}). Another reduction of (\ref{nonlocal}) is established considering the operator $B$ as a convolution integral
\begin{equation}
    (Bv)(x)= (\beta * v)(x)= \int_{\mathbb{R}} \beta(x-y) v(y) dy
\end{equation}
with the kernel function $\beta(x)$ and taking $L=B$. The resulting nonlocal nonlinear wave equation
\begin{equation}
    u_{tt}=[\beta * (u+g(u)) ]_{xx}
\end{equation}
describes the propagation of nonlinear strain waves in a one-dimensional, nonlocally elastic medium \cite{duruk1} (We refer the reader to \cite{erbay1, duruk2} for two different extensions of the model). Considering (\ref{nonlocal}) in such a general context provides several advantages. First, it gives results for the whole class of equation; for example our existence proof for traveling waves covers not only those given in the literature but all possible higher-order cases as well as nonlocal ones. Secondly, this generality allows us to observe behavior types, rather than particular phenomena; for instance we observe two different regimes in which traveling waves exist.

Existence and stability of traveling wave solutions of nonlinear wave equations are well studied in the literature starting from \cite{benjamin, bona1} (see \cite{pava1} for a recent overview of previous work). There have been a number of reliable existence, stability and instability results on the topic of solitary wave solutions of Boussinesq-type equations:  \cite{bona2, liu1, pego1, esfahani}. There are some studies addressing   similar issues for unidirectional nonlocal wave equations involving  pseudo-differential operators, see e.g.,   \cite{bona3, albert0, souganidis, albert1, albert2, albert3,  zeng, mats}. With specific forms of $L$ and $B$,  the same questions for the  nonlocal bidirectional wave equation  (\ref{nonlocal})    were studied in \cite{stubbe}. In that respect we should mention recent works \cite{bronski,hakkaev,stanislavova} where the authors developed a general theory to investigate spectral/linear stability of solitary wave solutions for Boussinesq-type equations similar to ours. The purpose of the present study is to investigate existence and stability properties of traveling waves for  the general class (\ref{nonlocal}). We emphasize that the present study does not require any homogeneity and similar assumptions  on the symbols $l(\xi )$ and $b(\xi )$.

 It is well known that wave velocity ranges of the solitary waves are different for the Boussinesq equation (\ref{gen-bouss}) and the improved Boussinesq equation (\ref{imp-boussinesq})  (for details, see the examples in Section 3). To summarize, the Boussinesq equation has solitary waves for small values of $c^{2}$ when $g(u)=- |u|^{p-1}u$, while the improved Boussinesq equation has solitary waves for large values of $c^{2}$ when $g(u)=|u|^{p-1}u$. In the present study, we first observe that this is a general phenomena; traveling wave solutions of the class (\ref{nonlocal}) with power nonlinearities  exist for two different regimes. In the first regime,  $c^{2}$ is small and $g(u)=- |u|^{p-1}u$ while in the second regime $c^{2}$ is large and $g(u)=|u|^{p-1}u$. Clearly, the Boussinesq equation and the improved Boussinesq equation are the most representative  and studied examples of these two regimes, respectively. In the case of power nonlinearities, $g(u)=\pm |u|^{p-1}u$, the traveling wave solutions $u=\phi _{c}(x-ct)$ of (\ref{nonlocal}) satisfy  the equation
\begin{equation}
    (L-c^{2}I)B^{-1}\phi_{c}\pm | \phi_{c} |^{p-1}\phi_{c}=0,   \label{ode}
\end{equation}
where $I$ is the identity operator. Then the order of $L$, i.e. $\rho$, is the determining parameter  in this distinction regarding (\ref{nonlocal}): for $\rho > 0$ the first regime occurs  and for $\rho < 0$ the second regime occurs. The case $\rho=0$ is of particular interest because both regimes occur. That is, when $\rho=0$,  traveling waves exist either for small $c^{2}$  and $g(u)=- |u|^{p-1}u$ or for large $c^{2}$ and $g(u)=|u|^{p-1}u$, as is observed for the double dispersion equation (\ref{gen-doubly}). In short,  $\rho$ determines the sign of $g(u)$ for which the traveling waves exist as well as the allowed values of $c$.  Therefore, in the sequel, we consider  the two regimes separately, which we will refer to shortly as the cases $\rho \geq 0$ and $\rho \leq 0$.

We first prove the existence of traveling wave solutions of (\ref{nonlocal}) for both  $\rho \geq 0$ and $\rho \leq 0$, separately. In both cases, the proof is based on a constrained variational problem, where traveling wave solutions appear as the critical points. We note that, in order to compensate for the non-homogeneity of the symbols, we use functionals that are not conserved integrals of (\ref{nonlocal}).  The concentration-compactness lemma of Lions \cite{lions1, lions2} is the main tool in establishing the existence of a minimizer of the constrained variational problem.  In the case of $\rho \geq 0$ the traveling wave solution is also a minimizer of a certain conserved quantity allowing us to go further. On the other hand, for  $\rho \leq 0$ the traveling wave solution turns out to be a saddle point and hence, as in the case of the improved Boussinesq equation, it does not allow us to get a stability result.

For orbital stability, in the case $\rho \geq 0$,  we adopt a well-known general criteria  in terms  of convexity of a certain function $d(c)$ related to conserved quantities. In particular cases of (\ref{nonlocal}), one can compute $d(c)$ explicitly, and obtain stability intervals for the wave velocity $c$. In our general case, this is not possible unless one makes further assumptions on the pseudo-differential operators $L$ and $B$. Nevertheless, we are able to show that for general $L$ and $B$ the function $d(c)$ is not convex when $c^{2}$ is sufficiently small. Moreover, for $c=0$ we further show the instability by blow-up using the blow-up threshold obtained in \cite{erbay2}.

Finally, we restrict our attention to the case  $L=I$ and general $B$. This gives rise to a class of Klein-Gordon-type equations.  We thus compute $d(c)$ explicitly, and obtain the orbital stability interval. Moreover, in this particular case, we are able to improve the blow-up result mentioned above for $c=0$ and obtain an interval of $c$ for instability by blow-up. It turns out that these two intervals complement one another. Hence,  for this class of Klein-Gordon-type equations, we have an almost complete characterization of stability/instability regions in terms of $c$. We note that our instability result is stronger than orbital instability as it involves blow-up in finite time. We believe that this is a new result even in the simpler case $B=(1-D_{x}^2)^{-1}$.

The structure of  the paper is as follows. Section 2 reviews some previously known results, including the local existence theorem and the conserved quantities for (\ref{nonlocal}). In Section 3, we start with some well-known examples that lead us to two regimes: $\rho \geq 0$  and $\rho \leq 0$. We then establish the existence of traveling wave solutions of (\ref{nonlocal}) in both regimes by introducing  constrained variational problems in a Sobolev  space setting and using the concentration-compactness lemma of Lions \cite{lions1, lions2}.  In Section 4, for the case  $\rho \geq 0$, we prove some orbital stability and instability by blow-up results for the traveling wave solutions of (\ref{nonlocal}).  In Section 5, for the case $L=I$, we provide an almost complete characterization of stability/instability regions.

The remaining part of this section is devoted to the notation that is used in the rest of the paper. Throughout the paper, the symbol $\widehat u$ represents the Fourier transform of $u$, defined by $\widehat u(\xi)=\int_\mathbb{R} u(x) e^{-i\xi x}dx$. The $L^p,~1\leq p <\infty$ and $L^{\infty}$ norms of $u$ on $\mathbb{R}$ are denoted by  $\Vert u\Vert_{L^p}$ and $ \Vert u\Vert_{L^\infty}$, respectively. The inner product of $u$ and $v$ in $L^2(\mathbb{R})$ is represented by $\langle u, v\rangle$. The $L^{2}$ Sobolev space of order $s$ on $\mathbb{R}$ is denoted by  $H^{s}=H^s(\mathbb{R})$ with the norm $\Vert u\Vert_{H^{s}}^2=\int_\mathbb{R} (1+\xi^2)^s |\widehat u(\xi)|^2 d\xi$.  The symbol $\mathbb{R}$ in $\int_{\mathbb{R}}$ will be mostly suppressed to simplify exposition. $C$ is a generic positive constant. $D_{x}$ is the partial derivative with respective to $x$.

\setcounter{equation}{0}
\section{Preliminaries: Local Existence and Conserved Quantities}
\noindent
In the study of existence and stability of traveling wave solutions of nonlinear dispersive equations both the local well-posedness theory of the inital-value problem and the conservation laws of energy and momentum play a key role. For the convenience of the reader, this section contains background material on these issues that will be used in later sections.

To make our exposition self-contained we start with the statement of the local existence theorem proved in \cite{babaoglu} for the Cauchy problem
\begin{eqnarray}
    && u_{tt}-Lu_{xx}=B(g(u))_{xx},~~~x\in \mathbb{R},~~~t>0  \label{nonlocal1} \\
    && u(x,0)=u_{0}(x),~~~u_{t}(x,0)=u_{1}(x),~~~x\in \mathbb{R} \label{nonlocal2}
\end{eqnarray}
with a general nonlinear function $g(u)\in C^{[s]+1}$ where $[.]$ denotes the ceiling function.
\begin{theorem}\label{theo2.1}\cite{babaoglu}
    Let $s>\frac{1}{2}$, $u_{0}\in H^{s}$, $u_{1}\in H^{s-1-\frac{\rho }{2}}$ and  $g\in C^{[s]+1}$. Assume that $L$ and $B$ satisfy (\ref{bn-a})-(\ref{bn-b}) with $\rho \geq 0$ and $r+\frac{\rho }{2}\geq 1$. Then, there exists some $T>0$ so that the Cauchy problem (\ref{nonlocal1})-(\ref{nonlocal2}) is locally well-posed with solution $u\in C([0,T),H^{s})\cap C^{1}([0,T),H^{s-1-\frac{\rho }{2}})$.
\end{theorem}
Before giving the conserved quantities, we make two remarks regarding Theorem \ref{theo2.1}.  First,  even though it was proved for $\rho \geq 0$ in \cite{babaoglu}, here we remark that the proof also works when $\rho >-2$. This is due to the acting semigroup
\begin{equation*}
    \mathcal{S}(t)v=\mathcal{F}^{-1}\left( \frac{\sin \left(t\xi \sqrt{l(\xi )}\right)}{\xi\sqrt{l(\xi )}}\right) \mathcal{F}v,
\end{equation*}
where $\mathcal{F}$ and $\mathcal{F}^{-1}$ are the Fourier and inverse Fourier transform operators. We note that $\rho +2$ is in fact the order of the operator $\partial_x^{2}L$. Observing that one may prove this new assertion in the same fashion as Theorem \ref{theo2.1} was proved, we leave the details to the reader. Secondly, when $\rho \leq -2$, (\ref{nonlocal1}) becomes an $H^{s}$-valued ordinary differential equation and then the local well-posedness proof of \cite{duruk1} applies. Below we state these two observations as a theorem:
\begin{theorem}\label{theo2.2}
    Let $s>\frac{1}{2}$, and  $g\in C^{[s]+1}$.
    \begin{enumerate}
    \item[(i)]  If $L$ and $B$ satisfy (\ref{bn-a})-(\ref{bn-b}) with $\rho >-2$ and $r+\frac{\rho }{2}\geq 1$, then there exists some $T>0$ so that the Cauchy problem (\ref{nonlocal1})-(\ref{nonlocal2}) is locally well-posed with solution $u\in C([0,T),H^{s})\cap C^{1}([0,T),H^{s-1-\frac{\rho}{2}})$ for initial data $u_{0}\in H^{s}$ and $u_{1}\in H^{s-1-\frac{\rho }{2}}$.
    \item[(ii)] If $L$ and $B$ satisfy (\ref{bn-a})-(\ref{bn-b}) with $~\rho \leq -2$ and $r\geq 2$,  then there exists some $T>0$ so that the Cauchy problem (\ref{nonlocal1})-(\ref{nonlocal2}) is locally well-posed with solution   $u\in C^{1}([0,T),H^{s})$  for initial data
        $u_{0}\in H^{s}$ and $u_{1}\in H^{s}$.
    \end{enumerate}
\end{theorem}
As it was done in \cite{erbay2}, for convenience we  rewrite  (\ref{nonlocal1}) as a system of  equations  and consider the Cauchy problem
\begin{eqnarray}
    && u_t=w_x,~~~x\in \mathbb{R},~~t>0  \label{sys1} \\
    && w_t= Lu_x+ B(g(u))_x,~~~x\in \mathbb{R},~~t>0  \label{sys2} \\
    && u(x,0)=u_{0}(x),~~~w(x,0)=w_{0}(x),~~x\in \mathbb{R}. \label{sys3}
\end{eqnarray}
Below we state the local well-posedness theorem of the Cauchy problem (\ref{sys1})-(\ref{sys3}) in terms of the pair $(u, w)$.
\begin{theorem}\label{theo2.3}
    Let $s>\frac{1}{2}$, and  $g\in C^{[s]+1}$.
    \begin{enumerate}
    \item[(i)]  If $L$ and $B$ satisfy (\ref{bn-a})-(\ref{bn-b}) with $\rho >-2$ and $r+\frac{\rho }{2}\geq 1$, then there exists some $T>0$ so that the Cauchy problem (\ref{sys1})-(\ref{sys3}) is locally well-posed with solution
        $(u,w)\in C([0,T),H^{s})\times C([0,T),H^{s-\frac{\rho}{2}})$ for initial data $(u_{0},w_{0})\in H^{s}\times H^{s-\frac{\rho }{2}}$.
    \item[(ii)] If $L$ and $B$ satisfy (\ref{bn-a})-(\ref{bn-b}) with $~\rho \leq -2$ and $r\geq 2$,  then there exists some $T>0$ so that the Cauchy problem (\ref{sys1})-(\ref{sys3}) is locally well-posed with solution
        $(u,w)\in C([0,T),H^{s})\times C([0,T),H^{s+1})$ for initial data $(u_{0},w_{0})\in H^{s}\times H^{s+1}$.
    \end{enumerate}
\end{theorem}
\begin{remark}\label{rem2.4}
Clearly, the solution predicted by Theorem \ref{theo2.3} can be extended to the maximal time interval $[0, T_{\max})$ where $T_{\max}$, if finite, is characterized by the blow-up conditions
\begin{displaymath}
    \limsup_{t \rightarrow T_{\max}^{-}} ~\left( \left\Vert u(t) \right\Vert_{s}+\left\Vert w(t)\right\Vert_{s-{\rho\over 2}} \right)= \infty  ~~~~\mbox{in case (i)}
\end{displaymath}
and
\begin{displaymath}
    \limsup_{t \rightarrow T_{\max}^{-}} ~\left( \left\Vert u(t) \right\Vert_{s}+\left\Vert w(t)\right\Vert_{s+1} \right)= \infty  ~~~~\mbox{in case (ii)}.
\end{displaymath}
\end{remark}
\begin{remark}\label{rem2.5}
In what follows we will be considering the power nonlinearity $g(u)=\pm |u|^{p-1}u$. Clearly $g\in C^{[s]+1}$ if and only if $p\geq [s]+1$ so that Theorems \ref{theo2.1}-\ref{theo2.3} apply for power nonlinearities when $p \geq [s]+1$. This restriction will be required only  in Theorem 4.8 and Theorem 5.1 where we consider a particular solution of the Cauchy problem.
\end{remark}
\begin{remark}\label{rem2.6}
 The above restriction on the pair $s,p$ may be weakened. The condition  $g\in C^{[s]+1}$ is used to prove that the map $u \rightarrow g(u)$ is locally Lipschitz on $H^{s}$. On the other hand, we believe that the same proof can be carried out by only  imposing the local Lipschitz condition  on the $[s]$ derivative of $g$, that is, $g\in C^{[s],0}$.
\end{remark}

The laws of conservation of energy and momentum for the system (\ref{sys1})-(\ref{sys3}) with $g(u)=\pm |u|^{p-1}u$ are
\begin{eqnarray}
    \mathcal{E}(u(t), w(t))&=&\frac{1}{2} \left\Vert B^{-1/2} w(t)\right\Vert^2_{L^{2}}
                    +\frac{1}{2} \left\Vert L^{1/2}B^{-1/2}u(t)\right\Vert^2_{L^{2}}
                    \pm \frac{1}{p+1}\left\Vert u(t)\right\Vert^{p+1}_{L^{p+1}}   \nonumber \\
                &=&\mathcal{E}(u_{0}, w_{0})  \label{energy}\\
    \mathcal{M}(u(t), w(t))&=&  \int \left( B^{-1/2} u(t)\right) \left(B^{-1/2} w(t)\right) dx=\mathcal{M}(u_{0}, w_{0}), \label{momentum}
\end{eqnarray}
respectively. For the details of deriving these conservation laws we refer the reader to \cite{erbay2}.

\setcounter{equation}{0}
\section{Existence of traveling waves}
\noindent

In this section we  prove that (\ref{nonlocal}) with $g(u)=\pm |u|^{p-1}u$, $p>1$  has traveling wave solutions of the form $u(x,t)=\phi_{c}(x-ct)$ for suitable values of  wave velocity $c$ and the appropriate choice of the sign $\pm$. Assuming that $\phi_{c}$, $LB^{-1}\phi_{c}$, $ B^{-1}\phi_{c}$ and their first-order derivatives decay sufficiently rapidly at infinity, it is readily seen that $ u(x,t)=\phi_{c}(x-ct)$ satisfies (\ref{nonlocal}) if $\phi_c$ solves (\ref{ode}). We will prove the existence of solutions of  (\ref{ode}) through a constrained variational problem.

To motivate our investigation we first consider the following three classical examples. \\
{\it Example 1. (The Boussinesq Equation)} If we take $L = I-\partial_{x}^{2}$, $B = I$ (for which $\rho = 2$ and $r = 0$, respectively) and $g(u)=-|u|^{p-1}u$, then (\ref{nonlocal}) reduces to the (generalized) Boussinesq equation \cite{bous}
    \begin{equation}
    u_{tt} - u_{xx} + u_{xxxx} = -(|u|^{p-1}u)_{xx}. \label{gen-bouss}
    \end{equation}
    Solitary wave solutions to the Boussinesq equation satisfy
    \begin{equation}
    \phi_{c}^{\prime\prime}-(1-c^{2})\phi_{c}+| \phi_{c} |^{p-1}\phi_{c}=0,  \label{gen-travel}
    \end{equation}%
    where the prime represents  the derivative with respect to $\zeta = x-ct$. When $c^{2} < 1$, the explicit solution  is given by
    \begin{equation}
    \phi_{c}(\zeta)=\left[{1\over 2}(p+1)(1-c^{2})\right]^{\frac{1}{p-1}}\mbox{sech} ^{\frac{2}{p-1}}\left[{1\over 2}(p-1)(1-c^{2})^{1\over 2}\zeta \right].  \label{gen-solution}
    \end{equation}%
{\it Example 2. (The Improved Boussinesq Equation)} If we take $L =B=(I-\partial_{x}^{2})^{-1}$ (for which $\rho = -2$ and $r = 2$ ) and $g(u)=|u|^{p-1}u$, then (\ref{nonlocal}) reduces to the improved Boussinesq equation \cite{ost}
    \begin{equation}
    u_{tt} - u_{xx} - u_{xxtt} = (|u|^{p-1}u)_{xx}. \label{imp-boussinesq}
    \end{equation}
    Solitary wave solutions to the improved Boussinesq equation satisfy
    \begin{equation}
    c^{2}\phi_{c}^{\prime\prime}-(c^{2}-1)\phi_{c}+| \phi_{c} |^{p-1}\phi_{c}=0.  \label{imp-travel}
    \end{equation}%
    When $c^{2} >1$, the explicit solution  is given by
    \begin{equation}
    \phi_{c}(\zeta)=\left[{1\over 2}(p+1)(c^{2}-1)\right]^{\frac{1}{p-1}}\mbox{sech}^{\frac{2}{p-1}}\left[{1\over {2}}(p-1)(1-{1\over c^{2}})^{1\over 2}\zeta \right].  \label{imp-solution}
    \end{equation}%
{\it Example 3. (The Double Dispersion Equation)}  Let  $L = (I-a_{1}\partial_{x}^{2})^{-1}(I-a_{2}\partial_{x}^{2})$ and $B = (I-a_{1}\partial_{x}^{2})^{-1}$ for two positive constants $a_{1}$ and $a_{2}$ in which $\rho = 0$ and $r =2$.  Then (\ref{nonlocal}) reduces to the double dispersion  equation \cite{samsonov, porubov}
    \begin{equation}
    u_{tt} -u_{xx}-a_{1}u_{xxtt}+a_{2}u_{xxxx}  =  (g(u))_{xx}. \label{gen-doubly}
    \end{equation}
    Solitary wave solutions to the double dispersion equation satisfy
    \begin{equation}
    (a_{2}-a_{1}c^{2})\phi_{c}^{\prime\prime}-(1-c^{2})\phi_{c}=g(\phi_{c}).  \label{dou-travel}
    \end{equation}%
     It is worth noting that $\mbox{sech} $-type solitary wave solutions to (\ref{dou-travel}) may be obtained  in two regimes. The first regime is identified by the equations
    \begin{equation}
    c^{2} <\min \{ 1, {a_{2}\over a_{1}}\}, ~~~~~ g( \phi_{c})=-|\phi_{c}|^{p-1}\phi_{c} \label{dd-min}
    \end{equation}
    and with the solitary wave solutions
    \begin{equation}
    \phi_{c}(\zeta)=\left[{1\over 2}(p+1)(1-c^{2})\right]^{\frac{1}{p-1}}
        \mbox{sech}^{\frac{2}{p-1}}\left[{1\over 2}(p-1)\left({{1-c^{2}}\over {a_{2}-a_{1}c^{2}}}\right)^{1\over 2}\zeta \right],  \label{double-solution-a}
    \end{equation}%
    whereas the second regime is described by
    \begin{equation}
    c^{2} >\max \{ 1, {a_{2}\over a_{1}}\}, ~~~~~ g( \phi_{c})=+|\phi_{c}|^{p-1}\phi_{c}  \label{dd-max}
    \end{equation}
    and with the solitary wave solutions
     \begin{equation}
    \phi_{c}(\zeta)=\left[{1\over 2}(p+1)(c^{2}-1)\right]^{\frac{1}{p-1}}
        \mbox{sech}^{\frac{2}{p-1}}\left[{1\over {2}}(p-1)\left({{c^{2}-1}\over {a_{1}c^{2}-a_{2}}}\right)^{1\over 2}\zeta \right].  \label{double-solution-b}
    \end{equation}%
   We note that the coercivity constants of $L$ in this particular case are
    \begin{displaymath}
    c_{1}^{2} =\min \{ 1, {a_{2}\over a_{1}} \}, ~~~~  c_{2}^{2} =\max \{ 1, {a_{2}\over a_{1}} \},
    \end{displaymath}
    hence the inequalities of (\ref{dd-min}) and  (\ref{dd-max}) can be expressed as  $c^{2} <  c_{1}^{2}$ and  $c^{2} >  c_{2}^{2}$, respectively.
    Note that, in the limiting cases $(a_{1}, a_{2})=(0,1)$ or $(a_{1}, a_{2})=(1,0)$, (\ref{gen-doubly}) reduces to the Boussinesq equation or the improved Boussinesq equation, respectively. Indeed,  in those limiting cases, one of the two regimes disappears.

As the above examples show, the sign of the  order of the operator $L$  and the sign of the nonlinear term determine together the range of $c$ for which  a traveling wave solution exists. The general case of (\ref{nonlocal}) can be handled in much the same way by identifying two regimes.  We describe the two regimes characterized by the equations
\begin{equation}
    \rho \geq 0, ~~~~ g( \phi_{c})=-|\phi_{c}|^{p-1}\phi_{c},  \label{case-gen}
\end{equation}
and by
\begin{equation}
    \rho \leq 0, ~~~~ g( \phi_{c})=+|\phi_{c}|^{p-1}\phi_{c},  \label{case-imp}
\end{equation}
respectively. While the Boussinesq equation serves as a prototype equation for the  case defined  in (\ref{case-gen}), the improved Boussinesq equation provides a prototype equation for the  case (\ref{case-imp}). In the same manner, we see that the double dispersion equation for which $\rho=0$ belongs to both of the two regimes.  In the next two subsections we will prove the existence of traveling wave solutions of (\ref{nonlocal}) for the regimes defined by (\ref{case-gen}) and (\ref{case-imp}), respectively.

\noindent
\subsection{The case $\rho\geq 0$ and $g(u)=-|u|^{p-1}u$  }

Throughout this subsection we assume that we are in the regime described by (\ref{case-gen}). To satisfy the requirements imposed by Theorem  \ref{theo2.1} we also assume that $L$ and $B$ satisfy (\ref{derivestimate})-(\ref{bn-b}) with  $r+{\rho \over 2}\geq 1$ in addition to $\rho \geq 0$. Let
\begin{equation}
    s_{0}={r \over 2}+{\rho \over 2}.
\end{equation}
Note that the above inequalities imply $s_{0}\geq {1\over 2}$. For $\psi \in H^{s_{0}}$, we now define the following functionals
\begin{eqnarray}
    &&\mathcal{I}_{c}(\psi )=\frac{1}{2}\int_{\mathbb{R}}(L^{1/2}B^{-1/2}\psi)^{2}dx
                            - \frac{c^{2}}{2}\int_{\mathbb{R}}(B^{-1/2}\psi)^{2}dx
    \label{ic0} \\
    &&\mathcal{Q}(\psi )=\int_{\mathbb{R}}|\psi |^{p+1}dx.  \label{momentum-t}
\end{eqnarray}%
It is worth pointing out that they are not conserved integrals of (\ref{nonlocal}). By the Sobolev embedding theorem, we have $H^{s_{0}}\subset H^{1/2}\subset L^{q}$ for all $q\geq 2$. This insures that the functionals $\mathcal{I}_{c}(\psi)$ and $\mathcal{Q}(\psi)$ are well-defined on $H^{s_{0}}$. We also note that the space $H^{s_{0}}\times H^{s_{0}-{\rho \over 2}}$ is the natural space for the energy and momentum functionals in (\ref{energy}) and (\ref{momentum}).

We begin by proving a coercivity estimate for $\mathcal{I}_{c}(\psi)$,  which holds only for $c^{2} < c_{1}^{2}$ where $c_{1}$ is the ellipticity constant of $L$.
\begin{lemma}\label{lem3.1}
    Let $c^{2} < c_{1}^{2}$. Then there are positive constants $\gamma _{1},\gamma_{2}$ such that
    \begin{displaymath}
        \gamma_{1}\Vert \psi \Vert_{H^{s_{0}}}^{2}\leq \mathcal{I}_{c}(\psi )\leq \gamma_{2}\Vert \psi \Vert_{H^{s_{0}}}^{2}.
    \end{displaymath}
\end{lemma}
\begin{proof}
    By (\ref{bn-a}) and  (\ref{bn-b}) we have
    \begin{equation}
        (c_{1}^{2}-c^{2})(1+\xi^{2})^{\rho / 2} \leq c_{1}^{2}(1+\xi^{2})^{\rho / 2}-c^{2}
            \leq l(\xi)-c^{2}\leq c_{2}^{2}(1+\xi^{2})^{\rho / 2},  \label{uuu1}
    \end{equation}
    and
    \begin{equation}
       {1\over  {c^{2}_{4}}}(1+\xi^{2})^{r / 2} \leq b^{-1}(\xi)\leq {1\over {c^{2}_{3}}}(1+\xi^{2})^{r / 2},  \label{uuu2}
    \end{equation}
    respectively. Using Parseval's theorem for (\ref{ic0}) and combining
    \begin{displaymath}
        \mathcal{I}_{c}(\psi )=\frac{1}{2}\int \left( l(\xi)-c^{2}\right)b^{-1}(\xi) |\widehat{\psi}(\xi)|^{2}d\xi
    \end{displaymath}
    with (\ref{uuu1}) and (\ref{uuu2}) yields
    \begin{displaymath}
        {{c_{1}^{2}-c^{2}}\over {2c_{4}^{2}}}\Vert \psi \Vert_{H^{s_{0}}}^{2}\leq \mathcal{I}_{c}(\psi )
            \leq  {{c_{2}^{2}}\over {2c_{3}^{2}}}\Vert \psi \Vert_{H^{s_{0}}}^{2}.
    \end{displaymath}
\end{proof}
\begin{remark}\label{rem3.2}
The important point to note here is  that  the above proof works only under the assumption  $\rho \geq 0$.
\end{remark}
\begin{remark}\label{rem3.3}
    From Lemma \ref{lem3.1}  it follows that when $c^{2} < c^{2}_{1}$, $\sqrt{\mathcal{I}_{c}(\psi )}$ defines a norm equivalent to the $H^{s_{0}}$ norm.
\end{remark}
For $c^{2} < c^{2}_{1}$ we now consider the variational problem
\begin{equation}
    m_{1}(c)=\inf \left\{ \mathcal{I}_{c}(\psi ):\psi \in H^{s_{0}},~~\mathcal{Q}(\psi )=1\right\}.  \label{m1}
\end{equation}%
A sequence $\{\psi _{n}\}$ in $H^{s_{0}}$ is called a minimizing sequence for $m_{1}(c)$, if $\mathcal{Q}(\psi _{n})=1$ for all $n$ and $\lim\limits_{n\rightarrow \infty }\mathcal{I}_{c}(\psi_{n})=m_{1}(c)$. Let $\{\tilde{\psi}_{n}\}$ be a sequence in $H^{s_{0}}$ such that
 $\lim\limits_{n\rightarrow \infty }\mathcal{I}_{c}(\tilde{\psi}_{n})=m_{1}(c)$ and $\mathcal{Q}(\tilde{\psi}_{n})=\lambda_{n}$ with $\lim\limits_{n\rightarrow \infty }\lambda_{n}=1$. Then $\psi _{n}=\lambda _{n}^{-1/(p+1)}\tilde{\psi}_{n}$ will be a minimizing sequence and the sequences $\{\psi _{n}\}$ and $\{\tilde{\psi}_{n}\}$ have the same limiting behavior.  We will henceforth abuse the terminology and refer also to $\{\tilde{\psi}_{n}\}$ as a minimizing sequence.

We emphasize  here two aspects of the variational problem. First, $m_{1}(c)>0$. Since $\mathcal{Q}(\psi)=\Vert \psi \Vert_{L^{p+1}}^{p+1}=1$, we have $1= \Vert \psi\Vert_{L^{p+1}}^{p+1} \leq C \Vert \psi\Vert_{H^{s_{0}}}^{p+1}$  where $C$ is the Sobolev embedding constant. By Lemma \ref{lem3.1}, $\mathcal{I}_{c}(\psi )\geq \gamma_{1}\Vert \psi\Vert_{H^{s_{0}}}^{2}\geq \gamma_{1}C^{-1}>0$ so that $m_{1}(c)>0$. Second, note that a minimizing sequence $\{\psi _{n}\}$ is always bounded in $H^{s_{0}}$. This is a direct consequence of  $\Vert \psi_{n} \Vert_{H^{s_{0}}}^{2}\leq \gamma_{1}^{-1}\mathcal{I}_{c}(\psi_{n} )$ together with the fact that $\mathcal{I}_{c}(\psi_{n} )$ is convergent.

The main results of this subsection are Theorem \ref{theo3.11} establishing the existence of minimizers of (\ref{m1}) and Theorem \ref{theo3.13} showing that the minimizers are in fact traveling wave solutions of (\ref{nonlocal}). The rest of this section will be devoted mainly to the proof of Theorem \ref{theo3.11}, which is based on the Concentration Compactness Lemma of Lions \cite{lions1, lions2} given below.
\begin{lemma}\label{lem3.4} (Concentration Compactness Lemma)
    Let $\{\rho_n\}$ be a sequence of nonnegative functions in $L^{1}$ satisfying $\int \rho_n(x) dx=\mu$ for all $n$ and some $\mu>0$. Then there is a subsequence $\rho_{n_k}$ satisfying one of the following conditions:
    \begin{enumerate}
    \item[(i)] (Compactness) There are real numbers $y_k$ for $k=1,2,\cdots$, such that for any $\epsilon>0$, there is a $R>0$ large enough that
        \begin{displaymath}
            \int_{|x-y_k| \leq R} \rho_{n_k}(x) dx \geq \mu-\epsilon.
        \end{displaymath}
    \item[(ii)] (Vanishing) For any $R>0$, $\lim\limits_{k\rightarrow \infty} \sup\limits_{y\in \mathbb{R}}\int_{|x-y| \leq R} \rho_{n_k}(x) dx=0. $
    \item[(iii)] (Dichotomy) There exists $\tilde{\mu} \in (0,\mu)$ such that for any $\epsilon>0$, there exists $k_0\geq 1$, and $\rho^1_k$, $\rho^2_k\geq 0$ such that for $k\geq k_0$
        \begin{eqnarray*}
        && \|\rho_{n_k}-(\rho_k^1+\rho_k^2)\|_{L^{1}} \leq \epsilon, \\
        && \left| \int \rho_k^1(x) dx-\tilde{\mu} \right|\leq \epsilon,~~~~~ \left| \int \rho_k^2(x) dx-(\mu-\tilde{\mu}) \right|\leq \epsilon,~~~ \\
        && \mbox{supp}~ \rho_k^1\cap \mbox{supp}~ \rho_k^2 =\emptyset,~~~~~ \mbox{dist}\{ \mbox{supp}~ \rho_k^1, \mbox{supp}~ \rho_k^2\}\rightarrow \infty~    \mbox{as}~k\rightarrow \infty.
        \end{eqnarray*}
    \end{enumerate}
\end{lemma}
\begin{remark}\label{rem3.5}
Lemma \ref{lem3.4} also holds under the weaker condition $\lim_{n \rightarrow \infty} \int \rho_{n}(x) dx = \mu$  for some $\mu >0$.
\end{remark}
For later analysis, it will be convenient to express the functional $\mathcal{I}_{c}$ in the form
\begin{displaymath}
    \mathcal{I}_{c}(\psi )={1\over 2}\Vert K_{c}\psi \Vert_{L^{2}}^{2}+{1\over 2}\gamma_{c}\Vert \psi \Vert_{L^{2}}^{2}
\end{displaymath}
where $K_{c}$ is a suitable coercive operator with the symbol $k_{c}(\xi)$ and $\gamma_{c}$ is a positive constant. This is equivalent to saying that
\begin{displaymath}
    (L-c^{2}I)B^{-1}=K_{c}^{2}+\gamma_{c}I
\end{displaymath}
or, in terms of the symbols $ \left(l(\xi)-c^{2}\right)b^{-1}(\xi)=k_{c}^{2}(\xi)+\gamma_{c}$.
By (\ref{uuu1}) and (\ref{uuu2}) it is obvious that $(l(\xi)-c^{2})b^{-1}(\xi)\geq (c_{1}^{2}-c^{2})c_{4}^{-2}$. So taking $\gamma_{c}=(c_{1}^{2}-c^{2})/(2c_{4}^{2})$ we get
\begin{displaymath}
    k_{c}^{2}(\xi)=(l(\xi)-c^{2})b^{-1}(\xi)-{{c_{1}^{2}-c^{2}}\over {2c_{4}^{2}}}\geq {{c_{1}^{2}-c^{2}}\over {2c_{4}^{2}}}.
\end{displaymath}
Clearly $K_{c}$ is a pseudo-differential operator of order $s_{0}$, exhibiting decay properties similar to those in (\ref{derivestimate}).

Let the sequence $\{\rho_{n}(x)\}$ be defined by
\begin{displaymath}
    \rho_{n}(x)={1\over 2} \vert K_{c}\psi_{n}(x)\vert^{2} +{1\over 2}\gamma_{c} \vert \psi_{n}(x) \vert^{2}
\end{displaymath}
for a minimizing sequence $\{\psi _{n}\}$.  By the definition of a minimizing sequence we have $\lim_{n \rightarrow \infty} \int \rho_{n}dx=m_{1}(c)>0$. In what follows, we will apply the concentration-compactness principle of Lions to the above-defined sequence $\rho_{n}$. We follow the classical approach and show that neither vanishing nor dichotomy holds. To this end, we have divided our task into a sequence of lemmas. To rule out vanishing we will use the following lemma \cite{pava1} (pp 125), which is a variant of Lemma I.1 in \cite{lions2}:
\begin{lemma}\label{lem3.6}
    Suppose $\alpha >0$ and $\delta >0 $ are given. Then there exists $\eta=\eta(\alpha, \delta)>0$ such that if $f_{n}\in H^{1/2}$ with $\Vert f_{n} \Vert_{H^{1/2}}\leq \alpha$ and $\Vert f_{n} \Vert_{L^{p+1}}\geq \delta$, then
    \begin{displaymath}
    \lim_{n \rightarrow \infty}\sup_{y\in\mathbb{R}} \int_{y-2}^{y+2} \mid f_n(x) \mid^{p+1} dx \geq \eta.
    \end{displaymath}
\end{lemma}
We can now state and prove the following.
\begin{lemma}\label{lem3.7}
    Vanishing does not occur.
\end{lemma}
\begin{proof}
    We proceed by contradiction and assume that  vanishing occurs. Then
    \begin{displaymath}
    \lim_{k \rightarrow \infty}\sup_{y\in\mathbb{R}} \int_{y-2}^{y+2} \mid \psi_{n_{k}}(x) \mid^{2} dx =0.
    \end{displaymath}
    Since $\psi_{n_{k}}$ is bounded in $H^{s_{0}}\subset H^{1/2}$, we have $\Vert  \psi_{n_{k}} \Vert_{H^{1/2}} \leq \alpha$   and $\Vert \psi_{n_{k}} \Vert_{L^{p+1}}=1$. It follows from Lemma \ref{lem3.6} that there is some $\eta >0$ for which
     \begin{displaymath}
   \lim_{k \rightarrow \infty} \sup_{y\in\mathbb{R}} \int_{y-2}^{y+2} \mid \psi_{n_{k}}(x) \mid^{p+1} dx \geq \eta.
    \end{displaymath}
    On the other hand,
    \begin{eqnarray*}
    \left( \int_{y-2}^{y+2} \mid \psi_{n_{k}}(x) \mid^{p+1} dx \right)^{2}
     & \leq & \left(\int_{y-2}^{y+2} \mid \psi_{n_{k}}(x) \mid^{2p} dx \right)   \left(\int_{y-2}^{y+2} \mid \psi_{n_{k}}(x) \mid^{2} dx \right)\\
     & \leq &   \Vert  \psi_{n_{k}} \Vert_{L^{2p}}^{2p} \int_{y-2}^{y+2} \mid \psi_{n_{k}}(x) \mid^{2} dx  \\
     & \leq &   C \Vert  \psi_{n_{k}} \Vert_{H^{1/2}}^{2p} \int_{y-2}^{y+2} \mid \psi_{n_{k}}(x) \mid^{2} dx  \\
     & \leq &   C \alpha^{2p} \int_{y-2}^{y+2} \mid \psi_{n_{k}}(x) \mid^{2} dx,
    \end{eqnarray*}
    which implies
    \begin{displaymath}
     \eta^{2} \leq \lim_{k \rightarrow \infty}\sup_{y\in\mathbb{R}} \left(\int_{y-2}^{y+2} \mid \psi_{n_{k}}(x) \mid^{p+1} dx\right)^{2} \leq
     C \alpha^{2p} \lim_{k \rightarrow \infty}\sup_{y\in\mathbb{R}} \int_{y-2}^{y+2} \mid \psi_{n_{k}}(x) \mid^{2} dx.
    \end{displaymath}
   This contradicts our assumption.
\end{proof}

To prove that dichotomy does not occur, it is convenient to define the family of variational problems
\begin{equation}
    m_{\lambda}(c)=\inf \left\{ \mathcal{I}_{c}(\phi ):\phi \in H^{s_{0}},~~\mathcal{Q}(\phi )=\lambda \right\}  \label{var2}
\end{equation}%
where $\lambda > 0$. Note that as $\mathcal{I}_{c}$ and $\mathcal{Q}$ are homogeneous of degrees 2 and $p+1$, respectively, we have the scaling $m_{\lambda}(c) = \lambda^{2\over {p+1}}m_{1}(c)$. Moreover, since $g(\eta) = \eta^{2\over {p+1}}+(1-\eta)^{2\over {p+1}}>1$ for all $\eta \in \left(0, 1\right)$, we obtain the strict subadditivity condition of $m_{\lambda}(c)$ described in the following lemma:
\begin{lemma}\label{lem3.8}
    For any $\lambda \in \left(0, 1\right)$,
    \begin{displaymath}
    m_{\lambda}(c)+m_{1-\lambda}(c)>m_{1}(c).
    \end{displaymath}
\end{lemma}
  We need commutator estimates for pseudo-differential operators  to control nonlocal terms. The following lemma is due to \cite{zeng} (Lemma 2.12). Below we give an alternative proof relying, as in \cite{zeng}, on the commutator  estimate of Coifman and Meyer (Theorem 35 of \cite{coifman}). We note that for $N=s_{0}=0$ the assertion of  Lemma \ref{lem3.9} reduces to Coifman and Meyer's estimate.
\begin{lemma}\label{lem3.9}
    Let $u\in H^{s_{0}}$ and $\theta\in C^{\infty}(\mathbb{R})$ with bounded derivatives of all orders.  Then, for the commutator $  \left[K_{c},\theta \right]u=K_{c}(\theta u)-\theta K_{c}u$ we have the estimate
    \begin{displaymath}
        \Vert \left[K_{c},\theta \right]u \Vert_{L^2}
            \leq C \left( \sum_{n=1}^{N+1} \Vert \theta^{(n)}\Vert_{L^\infty} \right) \Vert u\Vert_{H^{s_{0}}},
    \end{displaymath}
    where $N=[s_{0}]$ and $C$ is a positive constant.
\end{lemma}
\begin{proof}
    Before embarking on the proof, let us write down $k_{c}(\xi )$ in the form:
    \begin{displaymath}
        k_{c}(\xi )=k_{c}(0)+\sum_{j=1}^{N}\frac{k_{c}^{(j)}(0)}{j!}\xi^{j}+\xi^{N+1}r(\xi )
    \end{displaymath}%
    where a superscript in parenthesis indicates order of the derivative. We thus get $K_{c}=k_{c}(0)I+P(D_{x})+D_{x}^{N+1}R$, where $P(D_{x})$ is  the differential operator of order $N$ with vanishing constant term and $R$ is the operator with symbol $r(\xi)$ of nonpositive order.  Also we have the decay estimates
    \begin{displaymath}
        |D_{\xi}^{n}r(\xi)|={\cal O}(| \xi |^{-n})  ~~~\mbox{as} ~~ | \xi | \rightarrow \infty ~~\mbox{for every} ~~n \in \mathbb{N}.
    \end{displaymath}%
    Hence $R$ satisfies the hypotheses of Theorem 35 in \cite{coifman} and thus there exists a constant $C$ such that
    \begin{displaymath}
    \Vert [ R,\theta ]f^{\prime }\Vert_{L^{2}}\leq C\Vert \theta^{\prime}\Vert_{L^{\infty }}\Vert f\Vert_{L^{2}}.
    \end{displaymath}
    An easy computation shows that the commutator satisfies
    \begin{equation}
    [ K_{c}, \theta ]=[P(D_{x}),\theta ]+[RD_{x}^{N+1},\theta ].  \label{ltheta}
    \end{equation}
     Note that $D_{x}$ commutes with $R$. By the Leibniz rule we have
    \begin{displaymath}
        [ P(D_{x}),\theta ]u=P(D_{x})(\theta u)-\theta P(D_{x})u=\sum_{n=1}^{N}\theta^{(n)}P_{N-n}(D_{x})u,
    \end{displaymath}
    where $P_{N-n}(D_{x})$ is a differential operator of order $N-n$. We thus get
    \begin{eqnarray}
        \Vert [ P(D_{x}),\theta ]u\Vert_{L^{2}}
            &\leq & \sum_{n=1}^{N}\Vert \theta^{(n)}P_{N-n}(D_{x})u\Vert_{L^{2}}
            \leq C\left( \sum_{n=1}^{N}\Vert \theta^{(n)}\Vert_{L^{\infty }}\Vert D_{x}^{N-n}u\Vert_{L^{2}}\right)  \nonumber \\
            &\leq &C\left( \sum_{n=1}^{N}\Vert \theta ^{(n)}\Vert_{L^{\infty }}\right)\Vert u\Vert_{H^{N}}.  \label{ptheta}
    \end{eqnarray}%
    Using the Leibniz rule again we obtain
    \begin{eqnarray*}
        [ RD_{x}^{N+1},\theta ]u &=& RD_{x}^{N+1}(\theta u)-\theta (RD_{x}^{N+1}u)
            =R\left(\sum_{n=0}^{N+1}C_{N+1}^{n}\theta^{(n)}D_{x}^{N+1-n}u\right) -\theta (RD_{x}^{N+1}u) \\
        &=&\sum_{n=1}^{N+1}C_{N+1}^{n}R(\theta^{(n)}D_{x}^{N+1-n}u)+R(\theta D_{x}^{N+1}u)-\theta (RD_{x}^{N+1}u) \\
        &=&\sum_{n=1}^{N+1}C_{N+1}^{n}R(\theta^{(n)}D_{x}^{N+1-n}u)+[R,\theta ]D_{x}^{N+1}u,
    \end{eqnarray*}%
    where the $C_{N+1}^{n}$'s are constants.     We proceed to show that
    \begin{eqnarray}
        \left\Vert \sum_{n=1}^{N+1}C_{N+1}^{n}R\left(\theta^{(n)}D_{x}^{N+1-n}u\right)\right\Vert_{L^{2}}
            &\leq &\sum_{n=1}^{N+1}C_{N+1}^{n}\Vert R(\theta^{(n)}D_{x}^{N+1-n}u)\Vert_{L^{2}} \nonumber \\
            &\leq & C \sum_{n=1}^{N+1}\Vert \theta^{(n)}\Vert_{L^{\infty }}
                \Vert D_{x}^{N+1-n}u\Vert_{L^{2}} \nonumber \\
            &\leq & C\left( \sum_{n=1}^{N+1}\Vert \theta^{(n)}\Vert_{L^{\infty}}\right) \Vert u\Vert_{H^{N}}. \label{yyy}
    \end{eqnarray}%
    By Coifman and Meyer's theorem \cite{coifman} it follows that
    \begin{equation}
        \Vert [ R,\theta ]D_{x}^{N+1}u\Vert_{L^{2}}=\Vert [ R,\theta](D_{x}^{N}u)^{\prime }\Vert_{L^{2}}
            \leq C\Vert \theta^{\prime }\Vert_{L^{\infty }}\Vert D_{x}^{N}u\Vert_{L^{2}}
            \leq C\Vert \theta^{\prime }\Vert_{L^{\infty }}\Vert u\Vert _{H^{N}}.  \label{ttheta}
\end{equation}%
Finally, combining (\ref{ltheta}), (\ref{ptheta}), (\ref{yyy}) and (\ref{ttheta})  yields the result.
\end{proof}
Next, we rule out dichotomy through the following lemma.
\begin{lemma}\label{lem3.10}
    Dichotomy does not occur.
\end{lemma}
\begin{proof}
    Suppose dichotomy occurs. Then, by Lemma \ref{lem3.4}  there is $\tilde \mu \in (0,\mu)$ such that for any $\epsilon>0$, there exists $k_0\geq 1$, and $\rho^1_k, \rho^2_k\geq 0$ such that for $k\geq k_0$
    \begin{eqnarray*}
        && \|\rho_{n_{k}}-(\rho_k^1+\rho_k^2)\|_{L^1} \leq \epsilon, \\
        && | \int \rho_k^1 dx-\tilde\mu |\leq \epsilon,~~~ | \int \rho_k^2 dx-(\mu-\tilde \mu) |\leq \epsilon,~~~ \\
        && \mbox{supp}~ \rho_k^1\cap \mbox{supp}~ \rho_k^2 =\emptyset,~~ \mbox{dist}
            \{ \mbox{supp}~ \rho_k^1, \mbox{supp}~ \rho_k^2\}\rightarrow \infty,~\mbox{as}~k\rightarrow \infty.
    \end{eqnarray*}
    As in Lions \cite{lions1}, assume that the supports of $\rho_k^1$ and $\rho_k^2$ are of the form:
    \begin{displaymath}
        \mbox{supp}~ \rho_k^1\subset (y_k-R_k, y_k+R_k),~~~~~ \mbox{supp}~\rho_k^2\subset (-\infty, y_k-2 R_k) \cup(y_k+2 R_k, \infty)
    \end{displaymath}
    for some $R_k \rightarrow \infty$. Thus we have for $k\geq k_0$
    \begin{displaymath}
        \int_{R_k\leq \mid x-y_k\mid \leq 2R_k} \rho_{n_k} dx \leq \|\rho_{n_k}-(\rho_k^1+\rho_k^2)\|_{L^1}\leq \epsilon.
    \end{displaymath}
    We now choose a function $\theta ^{1}(x)\in C^{\infty }(\mathbb{R})$ so that $0\leq \theta^{1}\leq 1$.  Let $\theta^{1}(x)=1$ for $|x|\leq 1$ and $\theta^{1}(x)=0$ for $|x|\geq 2$. Let $\theta^{2}(x)$ be defined by $\theta ^{2}(x)=1-\theta ^{1}(x)$.  Define $\theta_{k}^{i}(x)=\theta^{i}(\frac{x-y_{k}}{R_{k}})$ and  $\psi_{k}^{i}(x)=\theta_{k}^{i}(x)\psi_{n_{k}}(x)$ for $i=1,2$. Hence we have $\psi_{n_{k}}(x)=\psi_{k}^{1}(x)+\psi_{k}^{2}(x)$ and
    \begin{equation}
        \mathcal{I}_{c}(\psi _{n_{k}})
            =\mathcal{I}_{c}(\psi_{k}^{1})+\mathcal{I}_{c}(\psi_{k}^{2})
                +\int(K_{c}\psi_{k}^{1})\left( K_{c}\psi_{k}^{2}\right) dx
                +\gamma_{c}\int\psi_{k}^{1}\psi_{k}^{2}dx.  \label{Ic}
    \end{equation}
    We first rewrite the first integral term as follows:
    \begin{eqnarray*}
        \int (K_{c}\psi _{k}^{1})\left( K_{c}\psi _{k}^{2}\right) dx
        &=&\int (K_{c}\theta _{k}^{1}\psi _{n_{k}})\left( K_{c}\theta _{k}^{2}\psi_{n_{k}}\right) dx \\
        &=&\int \left\{ \theta _{k}^{1}K_{c}\psi _{n_{k}}+[K_{c},\theta_{k}^{1}]\psi_{n_{k}}\right\}
                \left\{ \theta_{k}^{2}K_{c}\psi_{n_{k}}+[K_{c},\theta_{k}^{2}]\psi _{n_{k}}\right\} dx \\
        &=&\int \left\{ \theta _{k}^{1}\theta_{k}^{2}\left( K_{c}\psi_{n_{k}}\right)^{2}
            +\left([K_{c},\theta_{k}^{1}]\psi_{n_{k}}\right) \left([K_{c},\theta_{k}^{2}] \psi _{n_{k}} \right) \right. \\
        &&~~~~~~  \left.  +\left( \theta_{k}^{1}[K_{c},\theta _{k}^{2}]\psi_{n_{k}}
                +\theta_{k}^{2}[K_{c},\theta_{k}^{1}]\psi_{n_{k}}\right) K_{c}\psi_{n_{k}}
            \right\} dx .
    \end{eqnarray*}
    For large $k$ we estimate
    \begin{displaymath}
        \int \theta_{k}^{1}\theta_{k}^{2}(K_{c}\psi_{n_{k}})^{2}dx
        \leq     \int_{R_{k}\leq |x-y_{k}|\leq 2R_{k}}(K_{c}\psi_{n_{k}})^{2}dx
        \leq     \int_{R_{k}\leq |x-y_{k}|\leq 2R_{k}}\rho_{n_{k}}dx\leq \epsilon .
    \end{displaymath}%
    Note that we have
    \begin{eqnarray*}
        \int \left([K_{c},\theta_{k}^{1}]\psi_{n_{k}}\right) \left([K_{c},\theta_{k}^{2}]\psi_{n_{k}}\right) dx
        &\leq &\Vert \lbrack K_{c},\theta_{k}^{1}]\psi_{n_{k}}\Vert_{L^{2}}\Vert \lbrack K_{c},\theta_{k}^{2}]\psi _{n_{k}}\Vert _{L^{2}}, \\
        \int\left( \theta_{k}^{1}[K_{c},\theta_{k}^{2}]\psi_{n_{k}}+\theta_{k}^{2}[K_{c},\theta_{k}^{1}]\psi _{n_{k}}\right) K_{c}\psi_{n_{k}}dx
        &\leq &\Vert K_{c}\psi_{n_{k}}\Vert_{L^{2}}\left( \Vert \lbrack K_{c},\theta_{k}^{1}]\psi_{n_{k}}\Vert_{L^{2}}+\Vert \lbrack K_{c},\theta_{k}^{2}]\psi_{n_{k}}\Vert_{L^{2}}\right) ,
    \end{eqnarray*}
    By the commutator estimate of Lemma \ref{lem3.9},  we get
    \begin{displaymath}
        \Vert [ K_{c},\theta_{k}^{i}]\psi_{n_{k}}\Vert_{L^{2}}\leq C\left(\sum_{n=1}^{N+1}\Vert \theta_{k}^{i(n)}\Vert_{L^{\infty}}\right) \Vert \psi_{n_{k}}\Vert_{H^{s_{0}}}\leq \frac{C^{i}}{R_{k}}
    \end{displaymath}
    for $i=1,2$.  Having disposed of the above results, we now return to the first integral term  in (\ref{Ic}). Thus, for large $k$ we have
    \begin{displaymath}
        \int (K_{c}\psi_{k}^{1})\left( K_{c}\psi_{k}^{2}\right) dx=\mathcal{O}(\epsilon ).
    \end{displaymath}
    The last integral term in (\ref{Ic}) can be handled similarly. From what has already been proved, we deduce that
    \begin{displaymath}
        \mathcal{I}_{c}(\psi_{n_{k}})=\mathcal{I}_{c}(\psi_{k}^{1})+\mathcal{I}_{c}(\psi_{k}^{2})+\mathcal{O}(\epsilon )+\mathcal{O}({1\over R_{k}}).
    \end{displaymath}
    Since $\epsilon >0 $ is arbitrary, it follows from (\ref{m1}) that
    \begin{equation}
        m_{1}(c)=\lim_{k\rightarrow \infty }\mathcal{I}_{c}(\psi_{n_{k}})
        \geq \lim_{k\rightarrow \infty }\inf \mathcal{I}_{c}(\psi_{k}^{1})
        +\lim_{k\rightarrow \infty }\inf \mathcal{I}_{c}(\psi _{k}^{2}). \label{mlambda}
    \end{equation}
    Since $\Vert \psi_{n_{k}}\Vert _{H^{s_{0}}}$ and $\Vert \psi_{n_{k}}\Vert _{L^{2p }}$ are uniformly bounded, we see that
    \begin{eqnarray*}
        \int (|\psi_{n_{k}}|^{p+1}-|\psi_{k}^{1}|^{p+1}-|\psi_{k}^{2}|^{p+1})dx
        &=&\int_{R_{k}\leq |x-y_{k}|
            \leq 2R_{k}}|\psi_{n_{k}}|^{p+1}|1-(\theta_{k}^{1})^{p+1}-(\theta_{k}^{2})^{p+1}|dx \\
        &\leq &\sup_{k}\Vert \psi_{n_{k}}\Vert_{L^{2p}}^{p}\left(\int_{R_{k}
            \leq|x-y_{k}|\leq 2R_{k}}|\psi _{n_{k}}|^{2}dx\right)^{1/2} \\
        &\leq &\sup_{k}\Vert \psi _{n_{k}}\Vert _{L^{2p }}^{p}\left(\int_{R_{k}
            \leq |x-y_{k}|\leq 2R_{k}}\rho _{n_{k}}dx \right)^{1/2}=\mathcal{O}(\epsilon ).
    \end{eqnarray*}
    Combining this with (\ref{momentum-t}) yields
    \begin{displaymath}
        1=\mathcal{Q}(\psi_{n_{k}})=\mathcal{Q}(\psi_{k}^{1})+\mathcal{Q}(\psi_{k}^{2})+\mathcal{O}(\epsilon ).
    \end{displaymath}
    By passing to a subsequence if necessary, we can assume that, for $i=1,2$,  $\lim_{k\rightarrow \infty }\mathcal{Q}(\psi_{k}^{i})=\lambda _{i}$  with $\lambda_{1}+\lambda _{2}=1$. Note that
    \begin{displaymath}
    \lim\limits_{k\rightarrow \infty }\inf\mathcal{I}_{c}(\psi_{k}^{i})\geq m_{\lambda_{i}}(c) ~~~ \mbox{for} ~~ i=1,2.
    \end{displaymath}
    We now show that $\lambda _{1}$ (and similarly $\lambda _{2}$) is non-zero. To this end, suppose $\lambda _{1}=0$. This gives $\lambda _{2}=1$ and $\lim\limits_{k\rightarrow \infty }\inf \mathcal{I}_{c}(\psi _{k}^{2})\geq m_{1}(c)$.    On the other hand, by the commutator estimates we have
    \begin{eqnarray*}
        \mathcal{I}_{c}(\psi_{k}^{1})
        & = & {1\over 2}\Vert K_{c}\psi_{k}^{1}\Vert_{L^{2}}^{2}+{1\over 2}\gamma_{c}\Vert \psi_{k}^{1}\Vert_{L^{2}}^{2} \\
        &\geq & {1\over 2}  \Vert \theta_{k}^{1}K_{c}\psi_{n_{k}}\Vert_{L^{2}}^{2}+ {1\over 2} \gamma_{c}\Vert \theta_{k}^{1}\psi_{n_{k}}\Vert_{L^{2}}^{2}
            -\Vert [K_{c},\theta_{k}^{1}]\psi_{n_{k}}\Vert_{L^{2}}\Vert K_{c}\psi_{n_{k}}\Vert_{L^{2}}    \\
        &\geq &\int \theta_{k}^{1}\rho_{n_{k}}dx-\mathcal{O}(\epsilon )\\
        &\geq &\int_{|x-y_{k}|\leq R_{k}}\rho_{n_{k}}dx-\mathcal{O}(\epsilon ) \\
        &\geq &\int_{|x-y_{k}|\leq R_{k}}\rho_{k}^{1}dx
            -\Vert \rho_{n_{k}}-(\rho_{k}^{1}+\rho_{k}^{2})\Vert_{L^{1}}-\mathcal{O}(\epsilon ),
    \end{eqnarray*}
    where we have used the fact that $\rho_{k}^{1}$ has support in $|x-y_{k}|\leq R_{k}$ and $\rho_{k}^{2} $ vanishes there. As $k\rightarrow \infty $ this yields
    \begin{displaymath}
        \lim_{k\rightarrow \infty }\inf \mathcal{I}_{c}(\psi_{k}^{1}) \geq \tilde{\mu},
    \end{displaymath}
    and by (\ref{mlambda}), we obtain $m_{1}(c)\geq  \tilde{\mu}+m_{1}(c)$, contradicting $\ \ \tilde{\mu}>0$. Then it follows that $\ \lambda _{i}\neq 0$ for $i=1,2$. We thus get
    \begin{displaymath}
        m_{1}(c)\geq m_{\lambda _{1}}(c)+m_{1-\lambda _{1}}(c)
    \end{displaymath}
   which contradicts the subadditivity property of Lemma \ref{lem3.8}. This completes the proof that the dichotomy does not occur.
\end{proof}
 So far, with Lemmas  \ref{lem3.7} and  \ref{lem3.10} we have ruled out the possibility of both vanishing and dichotomy. The Concentration-Compactness Lemma implies that "compactness" occurs. We are then in a position to prove the following theorem  establishing the existence of global minimizers.
\begin{theorem}\label{theo3.11}
    Assume that $\rho \geq 0$, $~r+{\rho \over 2}\geq 1$ and $c^{2}<c_{1}^{2}$. Let $\{\psi _{n}\}$ be a minimizing sequence for (\ref{m1}). Then there exists a subsequence $\{\psi_{n_{k}}\}$ and a sequence $\{y_{n_{k}}\}$ of real numbers such that $\psi _{n_{k}}(.+y_{n_{k}})$ converges to some $\psi \in H^{s_{0}}$ and $\psi$ is a minimizer for (\ref{m1}).
\end{theorem}
\begin{proof}
    Let $\{\psi _{n}\}$ be a minimizing sequence for (\ref{m1}). Since vanishing and dichotomy are ruled out, the concentration-compactness lemma  implies  that there is a subsequence $\{ \psi_{n_{k}} \}$ such that for any $\epsilon > 0$ there are $R > 0$ and real numbers $y_{k}$ satisfying
    \begin{displaymath}
        \int_{|x|\geq R} \mid \psi_{n_{k}}(x+y_{n_{k}})\mid^{2}dx < \epsilon.
    \end{displaymath}
    Since the sequence $\{\psi_{n}(.+y_{n_{k}})\}$ is bounded in $H^{s_{0}}$, replacing it by a subsequence if necessary, we can assume that it converges weakly to some $\psi \in H^{s_{0}}$. The tails of the functions $\psi_{n}(.+y_{n_{k}})$ are uniformly bounded by $\epsilon$ outside some interval $[-R,R]$ in the $L^{2}$ norm. $H^{s_{0}}([-R,R])$ is compactly embedded in $L^{2}([-R,R])$ so that $\psi_{n_{k}}(.+y_{n_{k}})$ restricted to $[-R,R]$ converges strongly to $\psi$ restricted to $[-R,R]$, in $L^{2}([-R,R])$. But then we have
    \begin{equation}
    \Vert \psi_{n_{k}}(.+y_{{n_{k}}})-\psi \Vert_{L^{2}}\leq \Vert \psi_{n_{k}}(.+y_{{n_{k}}})-\psi \Vert _{L^{2}([-R,R])}+2\epsilon.
    \end{equation}
    This shows that $\psi_{n_{k}}(.+y_{{n_{k}}})$ converges strongly to $\psi $ in $L^{2}$. Moreover, it follows from the embedding $H^{s_{0}}\subset L^{2p}$ that there is some $C>0$ so that $\Vert \psi_{n_{k}}(.+y_{{n_{k}}})\Vert_{L^{2p}} \leq C$ for all $n_{k}$. Then we have
    \begin{eqnarray*}
    \Vert \psi _{n_{k}}(.+y_{{n_{k}}})-\psi \Vert _{L^{p+1}}^{p+1}
        &\leq &\Vert\psi _{n_{k}}(.+y_{{n_{k}}})-\psi \Vert _{L^{2p}}^{p}\Vert \psi_{n_{k}}(.+y_{{n_{k}}})-\psi \Vert _{L^{2}}   \\
        &\leq & (2C)^{p}\Vert \psi _{n_{k}}(.+y_{{n_{k}}})-\psi \Vert _{L^{2}}.
        \nonumber
    \end{eqnarray*}%
    Hence $\psi_{n_{k}}(.+y_{{n_{k}}})$ also converges to $\psi \in L^{p+1}$ strongly and hence $Q(\psi )=1$. By the definition of $m_{1}(c)$, we get $\mathcal{I}_{c}(\psi )\geq m_{1}(c)$.  As it has already been stated in Remark \ref{rem3.3},  $\sqrt{\mathcal{I}_{c}(\psi )}$ defines a Hilbertian norm on $H^{s_{0}}$ equivalent to the standard norm. Denoting the corresponding inner product by $\langle . , . \rangle_{c}$ and recalling that $\psi_{n_{k}}(.+y_{n_{k}})$ is also a minimizing sequence, we get
   \begin{eqnarray*}
    \mathcal{I}_{c}(\psi ) =\langle \psi , \psi \rangle_{c}=\lim_{k \rightarrow \infty}  \langle \psi , \psi_{n_{k}}(.+y_{{n_{k}}}) \rangle_{c}
       &\leq & \lim_{k \rightarrow \infty }\sup \sqrt{\mathcal{I}_{c}(\psi)} \sqrt{\mathcal{I}_{c}(\psi_{n_{k}}(.+y_{n_{k}}))}    \\
       & =   & \sqrt{\mathcal{I}_{c}(\psi)} \sqrt{m_{1}(c)}
    \end{eqnarray*}%
    so that $\mathcal{I}_{c}(\psi ) \leq m_{1}(c)$. Combining with the reverse inequality above we obtain $\mathcal{I}_{c}(\psi )=m_{1}(c)$, so $\psi$ is the minimizer. This completes the proof.
\end{proof}
\begin{remark}\label{rem3.12}
    Note that in the above proof we have
    \begin{equation*}
    \mathcal{I}_{c}(\psi)=\lim_{k \rightarrow \infty }\mathcal{I}_{c}(\psi_{n_{k}}(.+y_{n_{k}})),
    \end{equation*}
    so the weak limit preserves the norm. Then it follows that it is a strong limit; in other words $\psi_{n_{k}}(.+y_{n_{k}})$ converges strongly to $ \psi \in H^{s_{0}}$.
\end{remark}
With Theorem \ref{theo3.11} in hand, we can now prove the following main result, namely, the existence of traveling wave solutions:
\begin{theorem}\label{theo3.13}
    Assume that $\rho \geq 0$ and $~r+{\rho \over 2}\geq 1$. Let $c^{2}<c_{1}^{2}$ and $g(u)=- |u|^{p-1}u$. Then  the traveling wave solutions of (\ref{nonlocal}) exist.
\end{theorem}
\begin{proof}
    The proof consists of two steps, first we show that a proper scaling of the minimizer is a weak solution of (\ref{ode}). Then applying a regularity argument, we deduce that this weak solution is actually strong and exhibits the necessary decay properties.     A minimizer $\psi \in H^{s_{0}}$ of the variational problem (\ref{m1}) is a weak solution of the Euler-Lagrange equation
    \begin{equation}
    (L-c^{2}I)B^{-1}\psi-\theta(p+1)| \psi |^{p-1}\psi=0,  \label{el-imp}
    \end{equation}
    where $\theta$ denotes a Lagrange multiplier. Multiplying (\ref{el-imp}) by $\psi $ and integrating gives  $2m_{1}(c)=\theta (p+1)$. Then
    \begin{displaymath}
    \phi_{c} =[2m_{1}(c)]^{1/(p-1)}\psi \in H^{s_{0}}
    \end{displaymath}
    is a weak solution of (\ref{ode}):
    \begin{equation}
    (L-c^{2}I)B^{-1}\phi_{c}- | \phi_{c}|^{p-1}\phi_{c}=0.  \label{ooo}
    \end{equation}
    As $s_{0}\geq \frac{1}{2}$ and $p>1$, we have $|\phi_{c} |^{p-1}\phi_{c} \in L^{2}$. Then, $(L-c^{2}I)^{-1}B$ is an operator of order $-(\rho +r)$, we get
     \begin{displaymath}
    \phi_{c}=(L-c^{2}I)^{-1}B(|\phi_{c} |^{p-1}\phi_{c} )\in H^{\rho +r}=H^{2s_{0}}.
    \end{displaymath}
    Thus $\phi_{c}$ is a strong solution of (\ref{ode}). We note that the regularity of $\phi_{c}$ may be improved: since $2s_{0}\geq 1$ so $\phi_{c} \in L^{\infty}$ and $D_{x}\phi_{c}\in L^{2}$. This in turn shows that $D_{x}(|\phi_{c} |^{p-1}\phi_{c} )=p|\phi_{c} |^{p-1}D_{x}\phi_{c}\in L^{2}$, implying that $|\phi_{c} |^{p-1}\phi_{c} \in H^{1}$. But then $\phi_{c} =(L-c^{2}I)^{-1}B(|\phi_{c} |^{p-1}\phi_{c} )\in H^{2s_{0}+1} \subset H^{2}$. This bootstrap argument can be repeated for larger $p$. In fact, when $p$ is odd,  $\phi_{c}\in C^{\infty}$.
\end{proof}

\noindent

\subsection{The case $\rho\leq 0$ and $g(u)=|u|^{p-1}u$  }
Throughout this subsection we assume that we are in the regime described by (\ref{case-imp}). In addition to $\rho \leq 0$ we also assume that either $\rho \leq -2$  and $r\geq 2$  or $\rho >-2 $ and $\frac{\rho }{2}+r\geq 1$. Under the assumption that $L$ and $B$ satisfy (\ref{derivestimate})-(\ref{bn-b}) the requirements of Theorem  \ref{theo2.3} are satisfied.  In what follows we take
\begin{displaymath}
    s_{0}=\frac{r}{2}.
\end{displaymath}
 The important point to note here is that $s_{0}\geq \frac{1}{2}$ for both sets of parameter values.  An immediate consequence of this fact is that the Sobolev embeddings of in the previous subsection also apply to the present case.

The crucial fact about $\mathcal{I}_{c}(\psi )$ for the present case is  that, when $\rho <0$, or when $\rho=0$ and $c^{2}$ is large, the term $\Vert B^{-1/2}\psi \Vert^{2}_{L^{2}}$ in (\ref{ic0}) dominates the others in $\mathcal{I}_{c}(\psi )$. Hence $\mathcal{I}_{c}(\psi )$  is no longer bounded from below. Nevertheless, we note that it is bounded from above for large values of $c^{2}$. This is due to the change in the sign of the nonlinear term.

Given the form of the nonlinear term, we look for a solution of the equation
\begin{equation}
    (L-c^{2}I)B^{-1}\phi_{c}+| \phi_{c} |^{p-1}\phi_{c}=0.  \label{negative}
\end{equation}%
We now define a new functional, $\mathcal{J}_{c}(\psi )$,   as the negative of  what we have considered above:
\begin{equation*}
    \mathcal{J}_{c}(\psi )=-\mathcal{I}_{c}(\psi ).
\end{equation*}
As a result, a new range of wave velocities is established to be able to prove a coercivity estimate for $\mathcal{J}_{c}(\psi )$.  The range is provided by the following lemma; the proof is very similar to that of Lemma \ref{lem3.1}.
\begin{lemma} \label{lem3.14}
    Let $c^{2}>c_{2}^{2}$. Then there are  positive constants $\gamma_{1}, \gamma_{2}$ such that
    \begin{displaymath}
    \gamma_{1}\Vert \psi \Vert_{H^{s_{0}}}^{2}\leq  \mathcal{J}_{c}(\psi )\leq \gamma_{2}\Vert \psi \Vert_{H^{s_{0}}}^{2}.
    \end{displaymath}
\end{lemma}
\begin{proof}
    From (\ref{bn-a}) we have
    \begin{displaymath}
        c^{2}-c_{2}^{2}\leq c^{2}-c_{2}^{2}(1+\xi^{2})^{\rho/2} \leq c^{2}-l(\xi) \leq c^{2}-c_{1}^{2}(1+\xi^{2})^{\rho/2}\leq c^{2}.
    \end{displaymath}
    Using this inequality and (\ref{bn-b}) with
    \begin{displaymath}
        \mathcal{J}_{c}(\psi ) = {1\over 2}\int \left(c^{2}-l(\xi )\right) b^{-1}(\xi ) |\widehat{\psi}(\xi)|^{2}d\xi
    \end{displaymath}
    gives
    \begin{displaymath}
        {{c^{2}-c_{2}^{2}}\over {2c_{4}^{2}}} \Vert \psi \Vert_{H^{s_{0}}}^{2}\leq \mathcal{J}_{c}(\psi )
            \leq {{c^{2}}\over {2c_{3}^{2}}} \Vert \psi \Vert_{H^{s_{0}}}^{2}.
    \end{displaymath}
\end{proof}
Accordingly we define a new variational problem as
\begin{equation}
    \tilde{m}_{1}(c)=\inf \{ \mathcal{J}_{c}(\psi ) : \psi \in H^{s_{0}},~~~\mathcal{Q}(\psi )=1 \}.  \label{iJ}
\end{equation}
The proof of the existence of a minimizer of $\tilde{m}_{1}(c)$ goes along the same lines as the proof of  that of $m_{1}(c)$ in the previous subsection. The only modification we need is in the decomposition of $\mathcal{J}_{c}(\psi )$. To this end, we express $\mathcal{J}_{c}(\psi )$ in the form
\begin{displaymath}
       \mathcal{J}_{c}(\psi )={1\over 2} \Vert \tilde{K}_{c}\psi \Vert^{2}+{1\over 2}\gamma_{c} \Vert \psi \Vert^{2}
\end{displaymath}
where $\tilde{K}_{c}$ is a suitable coercive operator with the symbol $\tilde{k}_{c}(\xi)$ and $\gamma_{c}$ is a positive constant again. This time the symbols satisfy
\begin{displaymath}
       \left(c^{2}-l(\xi )\right) b^{-1}(\xi )= \tilde{k}_{c}^{2}(\xi)+\gamma_{c}.
\end{displaymath}
By choosing $\gamma_{c}=(c^{2}-c_{2}^{2})/(2c_{4}^{2})>0$ we get
\begin{displaymath}
\tilde{k}_{c}^{2}(\xi)=\left(c^{2}-l(\xi )\right) b^{-1}(\xi)-{{c^{2}-c_{2}^{2}}\over {2c_{4}^{2}}}.
\end{displaymath}
It is clear that with this setting all the lemmas of the previous subsection will hold yielding the existence of minimizers $\tilde{m}_{1}(c)$.

Any minimizer $\psi$ of the variational problem (\ref{iJ}) solves the Euler-Lagrange equation
\begin{displaymath}
    (L-c^{2}I)B^{-1}\psi+\theta (p+1)| \psi |^{p-1}\psi=0,
\end{displaymath}
where $\theta$ is a Lagrange multiplier.  Then a function $\phi_{c}$ obtained by a suitable scaling of the minimizer $\psi$  will be a weak solution of (\ref{negative}). Applying the regularity argument in the proof of Theorem \ref{theo3.13} we obtain  its analogue:
\begin{theorem}\label{theo3.15}
    Assume that $\rho \leq 0$ and  that either $\rho \leq -2$  and $r\geq 2$  or $\rho >-2 $ and $\frac{\rho }{2}+r\geq 1$.   Let $c^{2}>c_{2}^{2}$ and $g(u)= |u|^{p-1}u$. Then the traveling wave solutions of (\ref{nonlocal}) exist.
\end{theorem}

\setcounter{equation}{0}
\section{Stability of traveling waves: The case $\rho\geq 0$ and $g(u)=-|u|^{p-1}u$}
\noindent

In this section we will discuss  stability  of traveling waves under the assumptions of Theorem  \ref{theo3.13}. The theorem  guarantees  that traveling waves exist for $c^{2}<c_{1}^{2}$. We will first consider orbital stability which roughly speaking, means that a  solution starting close to a traveling wave remains close to some possibly other traveling wave with the same velocity. As in \cite{levandosky}, we will prove that orbital stability occurs for a velocity $c$  if  a suitably defined  function $d$ is convex in a neighborhood of $c$. We then study the function $d(c)$ and show that it is not convex for small $c^{2}$, in other words, our method will not predict orbital stability for small $c^{2}$. Moreover, we show that the standing waves, $c=0$, are never orbitally stable. To be precise, we prove that for any standing wave we can find initial data arbitrarily close to it such that the corresponding solution of (\ref{nonlocal}) blows up in finite time.

Let $G_{c}$ denote the set of all traveling wave solutions $ \phi _{c}$ with a fixed wave velocity $c$ of (\ref{nonlocal}). We denote the corresponding set of solutions $\Phi_{c}=\left( \phi_{c}, -c\phi_{c}\right) $ of the system (\ref{sys1})-(\ref{sys2}) by
\begin{displaymath}
    {\mathcal{G}}_{c}=\left\{\Phi_{c}=(\phi_{c}, -c\phi_{c}):~\phi_{c}\in G_{c}\right\}.
\end{displaymath}
By Theorem \ref{theo2.3}, for a solution $U=(u,w)$ of the system (\ref{sys1})-(\ref{sys2}), we have $U(t)\in X=H^{s_{0}}\times H^{s_{0}-{\rho \over 2}}$. Hence, we will consider ${\mathcal{G}}_{c}$ as a subset of $X$. Notice that the space $X$ is endowed with the norm $\Vert U\Vert_{X}=\Vert u\Vert_{H^{s_{0}}}+ \Vert w\Vert_{H^{s_{0}-{\rho \over 2}}}$. We consider  orbital stability in the sense of $X-$ stability defined below.
\begin{definition}\label{def4.1}
The set $\mathcal{G}_{c}$ is said to be $X$-stable, if for any $\epsilon >0$ there exists some $\delta >0$ such that whenever
\begin{displaymath}
    \inf \left\{\Vert U_{0}-\Phi_{c}\Vert_{X} : \Phi_{c}\in \mathcal{G}_{c}\right\}<\delta,
\end{displaymath}
the solution $U(t)$ of the Cauchy problem (\ref{sys1})-(\ref{sys3}) with $U(0)=(u_{0}(x), w_{0}(x))$ exists for all $t>0$, and satisfies
\begin{displaymath}
    \sup_{t>0} \inf \left\{\Vert U(t)-\Phi_{c}\Vert_{X}:\Phi_{c}\in \mathcal{G}_{c}\right\}<\epsilon.
\end{displaymath}
\end{definition}
We recall that $\phi_{c}=\left[2m_{1}(c)\right]^{\frac{1}{p-1}}\psi_{c}$ where $\psi_{c}$ was the minimizer for $m_{1}(c)$. Then we get $\mathcal{Q }(\phi_{c} )=2\mathcal{I}_{c}(\phi_{c})=2^{\frac{p+1}{p-1}}\left[m_{1}(c)\right]^{\frac{p+1}{p-1}}$.
We begin by establishing the following relationship between the conserved quantities $\mathcal{E}$, $\mathcal{M}$ of Section 2 and the functionals $\mathcal{I}_{c}$, $\mathcal{Q }$ of Section 3.
\begin{lemma}\label{lem4.2}
    Every $\Phi_{c}\in \mathcal{G}_{c}$ is a minimizer for $\mathcal{E}(U) +c\mathcal{M}(U)$ with constraint
    \begin{displaymath}
    \mathcal{Q }(u )=2^{\frac{p+1}{p-1}}\left[m_{1}(c)\right]^{\frac{p+1}{p-1}}.
    \end{displaymath}
\end{lemma}
\begin{proof}
    Combining (\ref{energy})-(\ref{momentum}) with (\ref{ic0})-(\ref{momentum-t})  yields
    \begin{displaymath}
    \mathcal{E}(U) +c\mathcal{M}(U)=\frac{1}{2}\left\Vert B^{-1/2}\left( w+cu\right) \right\Vert_{L^{2}}^{2}
                +\mathcal{I}_{c}(u) -\frac{1}{p+1}\mathcal{Q}(u).
    \end{displaymath}
    Then
    \begin{equation}
    \mathcal{E}(U) +c\mathcal{M}(U)\geq \mathcal{I}_{c}(u) -\frac{1}{p+1}\mathcal{Q}(u)
                \geq \mathcal{I}_{c}(\phi_{c}) -\frac{1}{p+1}\mathcal{Q}(\phi_{c})
                =\mathcal{E}(\Phi_{c}) +c\mathcal{M}(\Phi_{c})  \label{em-iq}
    \end{equation}
    and the result follows.
\end{proof}
It is worth pointing out that  $\Phi_{c}$ is also a minimizer for $\mathcal{E}(U) + c\mathcal{M}(U)$ subject to the constraint
\begin{equation}
    \left\Vert L^{1/2}B^{-1/2}u\right\Vert_{L^{2}}^{2}-c^{2}\left\Vert B^{-1/2}u\right\Vert_{L^{2}}^{2}  -\Vert  u\Vert_{L^{p+1}}^{p+1}
        =2\mathcal{I}_{c}(u)-\mathcal{Q}(u)=0,~~~~u \neq 0
\end{equation}
(see \cite{erbay2} for more details).

We now define the function $d(c)$ by
\begin{equation}
    d(c) = \inf \left\{ \mathcal{E}(U) +c\mathcal{M}(U) :  U\in X, ~~~
        \mathcal{Q}(u)=2^{\frac{p+1}{p-1}}\left[m_{1}(c)\right]^{\frac{p+1}{p-1}}\right\}.
\end{equation}
From Lemma \ref{lem4.2} it follows that
\begin{displaymath}
    d(c) = \mathcal{E}(\Phi_{c}) +c\mathcal{M}(\Phi_{c}),
\end{displaymath}
or
\begin{equation}
    d(c)=\left(\frac{p-1}{p+1}\right)\mathcal{I}_{c}(\phi_{c})
            ={1\over 2}\left(\frac{p-1}{p+1}\right)\mathcal{ Q}(\phi_{c})
            =2^{\frac{2}{p-1}}\left(\frac{p-1}{p+1}\right)[m_{1}(c)]^{\frac{p+1}{p-1}}.
            \label{d3}
\end{equation}

\begin{lemma} \label{lem4.3}
    Suppose $d$ is differentiable; then $d^{\prime }(c)=\mathcal{M}(\Phi_{c})$.
\end{lemma}
\begin{proof}
    We have
    \begin{eqnarray*}
    d^{\prime }(c) &=&\frac{d}{dc}\int \left[ \frac{1}{2}\left(L^{1/2}B^{-1/2}\phi_{c}\right)^{2}
                     -\frac{c^{2}}{2}\left(B^{-1/2}\phi_{c}\right)^{2}
                     -\frac{1}{p+1}|\phi_{c}|^{p+1}\right] dx,   \\
     ~             &=& \int \left[\left(L-c^{2}I\right)B^{-1}\phi_{c}
                        -|\phi_{c}|^{p-1}\phi_{c}\right]\frac{d\phi_{c}}{dc}dx
                        -\int c\left(B^{-1/2}\phi_{c}\right)^{2}dx.
    \end{eqnarray*}
    Since $(L-c^{2}I)B^{-1}\phi _{c}-|\phi _{c}|^{p-1}\phi _{c}=0$ (see \ref{ooo}), we have the desired result;
    \begin{equation}
        d^{\prime }(c)=-\int c\left(B^{-1/2}\phi_{c}\right)^{2}dx=\mathcal{M}(\Phi_{c}). \label{dc1}
    \end{equation}
\end{proof}
As $\mathcal{M}(\Phi_{c})=-c \left\Vert B^{-1/2}\phi_{c}\right\Vert_{L^{2}}^{2}$, it follows from (\ref{dc1}) that, whenever differentiable on some interval not containing the origin,  the function $d(c)$ is monotone on the interval. We can state now the main result on orbital stability.
\begin{theorem}\label{theo4.4}
    Let $\rho \geq 0$, $r+{\rho \over 2}\geq 1$, $(\rho, r)\neq (0,1)$ and $c^{2}<c_{1}^{2}$. Suppose $d$ is differentiable and strictly convex on some interval $J$ containing $c$. Then the set $\mathcal{G}_{c}$ is $X-$stable.
\end{theorem}
\begin{proof}
    Suppose that $\mathcal{G}_{c}$ is $X-$unstable. Then there are some $\epsilon >0$,  initial data $U_{n}(0)$ and points $t_{n}>0$ such that
    \begin{displaymath}
    \inf_{\Phi_{c} \in \mathcal{G}_{c}}\Vert U_{n}(0)-\Phi_{c} \Vert _{X}<\frac{1}{n}     \label{t0}
    \end{displaymath}%
    but
    \begin{displaymath}
    \inf_{\Phi_{c} \in \mathcal{G}_{c}}\Vert U_{n}(t_{n})-\Phi_{c} \Vert_{X}\geq \epsilon,
    \end{displaymath}%
    where $U_{n}(t)=(u_{n}(t), w_{n}(t))$ is the solution of the Cauchy problem (\ref{sys1})-(\ref{sys3}) with $U_{n}(0)=(u_{n}(0), w_{n}(0))$. By continuity of $U_{n}(t)$ we can take $\epsilon $ sufficiently small and choose $t_{n}$ such that
    \begin{displaymath}
    \inf_{\Phi_{c} \in \mathcal{G}_{c}}\Vert U_{n}(t_{n})-\Phi_{c} \Vert_{X}=\epsilon .
    \end{displaymath}
    In addition to this, we also choose $\Phi_{c}^{n} \in \mathcal{G}_{c}$ such that
    \begin{displaymath}
    \lim_{n \rightarrow \infty}\Vert U_{n}(0)-\Phi_{c}^{n} \Vert_{X}=0 .
    \end{displaymath}
    Since the invariants $\mathcal{E}$ and $\mathcal{M}$ are continuous on $X$, we have
    \begin{eqnarray*}
    && \lim_{n\rightarrow \infty }\mathcal{E}(U_{n}(t_{n}))=\lim_{n\rightarrow \infty }\mathcal{E}(U_{n}(0))=\mathcal{E}(\Phi_{c}^{n}),      \\
    && \lim_{n\rightarrow \infty }\mathcal{M}(U_{n}(t_{n}))=\lim_{n\rightarrow \infty }\mathcal{M}(U_{n}(0))=\mathcal{M}(\Phi_{c}^{n}),
    \end{eqnarray*}%
    noting that the terms on the right-hand side are independent of $n$.
    By taking $\epsilon $ to be sufficiently small, we can make the values of $u_{n}(t_{n})$ arbitrarily close to $\phi_{c}^{n}$ and consequently the values of   $\mathcal{Q}(u_{n}(t_{n}))$  arbitrarily close to $\mathcal{Q}(\phi_{c}^{n})=2\left(\frac{p+1}{p-1}\right)d(c)$.   Since $d( c)$ is monotone on $J$, for each $n$, there is a unique $c_{n}$ satisfying
    \begin{displaymath}
     \mathcal{Q}( u_{n}( t_{n}) )=\mathcal{Q}( \phi_{c_{n}} )=2\left(\frac{p+1}{p-1}\right)d(c_{n}),
    \end{displaymath}
     for the traveling wave solution  $\phi_{c_{n}}$. This means  $\mathcal{Q}( u_{n}( t_{n}) )= \mathcal{Q}( \phi_{c_{n}} )=2^{\frac{p+1}{p-1}}[m_{1}(c_{n})]^{\frac{p+1}{p-1}}$. By Lemma \ref{lem4.2} we have
    \begin{equation}
    \mathcal{E}(U_{n}(t_{n}))+c_{n}\mathcal{M}(U_{n}(t_{n})) \geq d(c_{n}) . \label{em-dcn}
    \end{equation}
    On the other hand, we can write
    \begin{equation}
        d( c_{n}) =d( c) +d^{\prime }( c) (c_{n}-c) +\int_{c}^{c_{n}}\left[ d^{\prime}( s) -d^{\prime}( c) \right] ds. \label{dcn}
    \end{equation}
    By assumption, $d$ is strictly convex and consequently $d^{\prime }$ is strictly increasing. From this, it follows that the integral on the right-hand side is positive for $c \neq c_{n}$.   Using Lemma \ref{lem4.3}, we have
    \begin{eqnarray*}
        d( c) +d^{\prime }( c) ( c_{n}-c) &=& \mathcal{E}\left( \Phi_{c}^{n}\right) + c\mathcal{M}\left( \Phi_{c}^{n}\right) + \mathcal{M}\left( \Phi_{c}^{n}\right) ( c_{n}-c) \\
        ~                                 &=& \mathcal{E}\left( \Phi_{c}^{n}\right) + c_{n}\mathcal{M}\left( \Phi_{c}^{n}\right).
    \end{eqnarray*}
    Combining this with (\ref{em-dcn}) and (\ref{dcn}) yields
    \begin{displaymath}
        \mathcal{E}(U_{n}(t_{n}))+c_{n}\mathcal{M}(U_{n}(t_{n}))
            \geq \mathcal{E}\left( \Phi _{c}^{n}\right)+c_{n}\mathcal{M}\left( \Phi _{c}^{n}\right)
                +\int_{c}^{c_{n}}\left[ d^{\prime}(s) -d^{\prime}(c) \right] ds,
     \end{displaymath}
     or
     \begin{displaymath}
        \mathcal{E}(U_{n}(t_{n}))-\mathcal{E}\left( \Phi_{c}^{n}\right) +c_{n}\left( \mathcal{M}(U_{n}(t_{n})) -\mathcal{M}( \Phi_{c}^{n})\right)
                \geq  \int_{c}^{c_{n}}\left[ d^{\prime }( s) -d^{\prime}(c) \right] ds.
    \end{displaymath}
    But as $n\rightarrow \infty $, the left-hand side of the inequality converges to zero. As $d^{\prime }( s)$ is strictly  increasing this is possible only when $\lim_{n\rightarrow \infty }c_{n}=c$. Continuity of $ d$ implies that
    \begin{displaymath}
        \lim_{n\rightarrow \infty }\mathcal{Q}(u_{n}(t_{n}))
            =\lim_{n\rightarrow \infty}2\left(\frac{p+1}{p-1}\right)d(c_{n})= 2\left(\frac{p+1}{p-1}\right)d(c)=\mathcal{Q}(\phi_{c}^{n}).
    \end{displaymath}
    Taking the limit of both sides of the following inequality as $n\rightarrow \infty $
    \begin{displaymath}
     \mathcal{I}_{c}(u_{n}(t_{n}))-\frac{1}{p+1} \mathcal{Q}(u_{n}(t_{n}))
       \leq \mathcal{E}(U_{n}(t_{n}))+c \mathcal{M}(U_{n}(t_{n})),
    \end{displaymath}
    and using (\ref{d3}) we get
     \begin{displaymath}
    \lim_{n\rightarrow \infty }\mathcal{I}_{c}(u_{n}(t_{n}))\leq \lim_{n\rightarrow \infty } {2\over {p-1}}d(c_{n})+d(c)={{p+1}\over {p-1}}d(c)
    \end{displaymath}
    or
    \begin{displaymath}
    \lim_{n\rightarrow \infty }\mathcal{I}_{c}(u_{n}(t_{n}))\leq \mathcal{I}_{c}(\phi_{c}).
    \end{displaymath}
    This result implies that $\{u_{n}(t_{n})\}$ is a minimizing sequence. By the existence theorem of  traveling waves solutions, Theorem \ref{theo3.13}, there is a shifted subsequence  that converges in $H^{s_{0}}$ to some $\phi_{c}^{0}\in G_{c}$. We further note that
    \begin{displaymath}
    \frac{1}{2}\left\Vert B^{-1/2}\left( w_{n}(t_{n})+cu_{n}(t_{n})\right)\right\Vert_{L^{2}}^{2}
           =\mathcal{E}(U_{n}(t_{n}))+c\mathcal{M}(U_{n}(t_{n}))+\frac{1}{p+1}\mathcal{Q}(u_{n}(t_{n}))-\mathcal{I}_{c}(u_{n}(t_{n}))
    \end{displaymath}%
    converges to zero as $n\rightarrow \infty$. This gives $\lim_{n\rightarrow \infty }\left(w_{n}(t_{n})+cu_{n}(t_{n})\right)=0$  in $H^{s_{0}-\frac{\rho}{2}}$. Therefore, a shifted subsequence of $U_{n}(t_{n})$ converges in $X$ to $\Phi_{c}^{0}=(\phi _{c}^{0}, -c\phi _{c}^{0})$. In conclusion, we have
    \begin{displaymath}
    \inf_{\phi \in G_{c}}\Vert U_{n}(t_{n})-\Phi_{c} \Vert_{X}=0,
    \end{displaymath}%
    which contradicts our assumption. Note that $s_{0}={r \over 2}+{\rho \over 2}>{1\over 2}$ when $(\rho, r)\neq (0, 1)$. Hence  Theorem \ref{theo2.3} guarantees local well-posedness in $H^{s_{0}}\times H^{s_{0}-{\rho \over 2}}$. The above argument, at first attempt, can only hold locally, i.e. for $0 \leq t  <T$. On the other hand, the same argument  shows that $U(t)$ stays bounded in $H^{s_{0}}\times H^{s_{0}-{\rho \over 2}}$; hence can be  continued beyond $T$. This in fact shows that $U(t)$ is indeed global and stays close to the orbit for all times.
 \end{proof}
 \begin{remark}\label{rem4.5}
      In the case $(\rho, r)=(0, 1)$, namely, $s_{0}={1\over 2}$, the above proof shows that we have a weaker version of orbital stability in the following sense: If the initial data $U(0)\in H^{s}\times H^{s}$ (for some $s>{1\over 2}$) is close to the orbit in the weaker $H^{1\over 2}\times H^{1\over 2}$ norm,  then the  solution, as long as as it exists, remains close to the orbit in the same norm.
\end{remark}

 We now discuss convexity of $d(c)$. To this end we  investigate more closely the properties of $m_{1}(c)$. Let $M_{c}$ denote the set of minimizers for $m_{1}( c)$:
\begin{displaymath}
    M_{c}=\left \{\psi \in H^{s_{0}}:~ \mathcal{Q}(\psi )=1,~~~~\mathcal{I}_{c}(\psi)=m_{1}(c)\right\}.
\end{displaymath}
As $m_{1}( c)$ is an even function, it suffices to consider the interval $[0, c_{1})$.
\begin{lemma} \label{lem4.6}
    On the interval  $[0, c_{1})$ where $c_{1}$ is the coercivity constant of $L$, the following statements hold.
    \begin{enumerate}
    \item[(i)] The map $m_{1}(c) $ is strictly decreasing.
    \item[(ii)] The maps
        \begin{displaymath}
            \alpha^{-}( c) =\inf \left\{ \left\Vert B^{-{1/2}}\psi_{c}\right\Vert^{2}_{L^{2}} : \psi_{c}\in M_{c}\right\}, ~~~~
            \alpha^{+}( c) =\sup \left\{ \left\Vert B^{-{1/2}}\psi_{c}\right\Vert^{2}_{L^{2}} : \psi_{c}\in M_{c}\right\}
        \end{displaymath}%
    are strictly increasing.
    \item[(iii)] Except for countably many points, $\alpha^{-}( c) =\alpha^{+}( c) $ hence $\left\Vert B^{-{1/2}}\psi_{c}\right\Vert^{2}_{L^{2}}$ is constant on $M_{c}$.
    \item[(iv)] The map $m_{1}( c) $ is continuous on $[0,c_{1}) $, is differentiable and $m_{1}^{\prime}( c) =-c\left\Vert B^{-{1/ 2}}\psi_{c}\right\Vert^{2}_{L^{2}}$ at all points where $\alpha^{-}( c) =\alpha^{+}( c) $.
    \item[(v)] The map $m_{1}\left( c\right) $ is concave.
    \end{enumerate}
\end{lemma}
\begin{proof}
    Let $\tilde{c}\in \left[ 0,c_{1}\right) $ such that $c\not=\tilde{c}$. Suppose that $\psi_{c}$ and $\psi_{\tilde{c}}$ are two  minimizers corresponding to  $c$ and $\tilde{c}$, respectively. Then we have
    \begin{eqnarray*}
        m_{1}(c) &=&\mathcal{I}_{c}(\psi_{c}) =\frac{1}{2}\left\Vert L^{{1/2}}B^{-{1/2}}\psi_{c}\right\Vert_{L^{2}}^{2}
                -\frac{c^{2}}{2} \left\Vert B^{-{1/2}}\psi_{c}\right\Vert^{2}_{L^{2}} \\
        &=& \mathcal{I}_{\tilde{c}}\left( \psi_{c}\right) +\frac{\tilde{c}^{2}-c^{2}}{2} \left\Vert B^{-{1/2}}\psi_{c}\right\Vert^{2}_{L^{2}} \\
        &>& m_{1}( \tilde{c}) +\frac{\tilde{c}^{2}-c^{2}}{2}\left\Vert B^{-{1/2}}\psi_{c}\right\Vert^{2}_{L^{2}}.
    \end{eqnarray*}%
    By symmetry we get
    \begin{displaymath}
        \frac{\tilde{c}^{2}-c^{2}}{2}  \left\Vert B^{-{1/2}}\psi_{c}\right\Vert^{2}_{L^{2}}
            < m_{1}(c)-m_{1}(\tilde{c}) <\frac{\tilde{c}^{2}-c^{2}}{2}  \left\Vert B^{-{1/2}}\psi_{\tilde{c}}\right\Vert^{2}_{L^{2}}  .
    \end{displaymath}
    This proves assertions (i) and (ii) of the lemma. It also implies that $m_{1}( c) $ is continuous. From (ii) we conclude that $\alpha^{+}(c) $ and $\alpha^{-}( c) $ are continuous except for countably many points in $[0, c_{1})$. For (iii) notice that the intervals $ \left[ \alpha^{-}( c),  \alpha^{+}( c) \right] $ have disjoint interior; this is possible only if $ \alpha^{-}( c) =\alpha^{+}( c) $ except for countably many $c$, implying (iii). Take some $c$ where $\alpha^{-} $ is continuous and $ \alpha^{-}( c) =\alpha^{+}( c) $. For $c> \tilde{c}$,
    \begin{displaymath}
    -\frac{\tilde{c}+c}{2}\left\Vert B^{-{1/2}}\psi_{c}\right\Vert^{2}_{L^{2}}
            <\frac{m_{1}( c) -m_{1}( \tilde{c}) }{c-\tilde{c}}
            <-\frac{\tilde{c}+c}{2} \left\Vert B^{-{1/2}}\psi_{\tilde{c}}\right\Vert^{2}_{L^{2}},
    \end{displaymath}%
    with  the reverse inequality holding for $c<\tilde{c}$. Then
   \begin{displaymath}
        m_{1}^{\prime }( c) =\lim_{\tilde{c}\rightarrow c}\frac{m_{1}(c) -m_{1}(\tilde{c}) }{c-\tilde{c}}
                            =-c \left\Vert B^{-{1/2}}\psi_{c}\right\Vert^{2}_{L^{2}}
    \end{displaymath}
    as was predicted in Lemma \ref{lem4.3}.  Then, by assertion (ii), $m_{1}^{\prime }( c)$, whenever it exists, is strictly decreasing for $c>0$.  At the points where $m_{1}^{\prime }( c)$ does not exist we have corners  with the slopes  decreasing as we pass through the corners. Thus $m_{1}(c)$ is strictly concave. We also note that $m_{1}^{\prime }(0)=0$.
\end{proof}
We obtain from (\ref{d3})  that $d^{\prime}(c)=2^{\frac{2}{p-1}}[m_{1}(c)]^{\frac{2}{p-1}}m_{1}^{\prime }(c)$. Both $m_{1}(c)$ and $m_{1}^{\prime}(c)$ are decreasing for $c>0$. Since $m_{1}(c)>0$,  $m_{1}^{\prime }(0)=0$ and $m_{1}^{\prime}(c)<0$ we observe that $d^{\prime}$ decreases when $c$ is near zero. This means that $d(c)$ will not be convex for small $c$. Therefore, the stability result of Theorem \ref{theo4.4} will not apply to traveling waves with small velocity.  In fact, following the approach in \cite{liu2}, we now show that there is instability by blow up in the case $c = 0$. To that end we state Theorem 3.5 of \cite{erbay2} in the following form:
\begin{theorem}\label{theo4.7}
    Let $U_{0} = (u_{0}, w_{0})$ with $u_{0} = (v_{0})_{x}$ for some $v_{0}\in L^{2}$. Suppose $\mathcal{E}(U_{0}) < d(0)$ and $2\mathcal{I}_{0}(u_{0})-\mathcal{Q}(u_{0})< 0$. Then the solution $U(t)$ of the Cauchy problem (\ref{sys1})-(\ref{sys3}) with initial data $U_{0}$ blows up in finite time.
\end{theorem}
Using Theorem \ref{theo4.7}, we now prove that the set of standing waves, $\mathcal{G}_{0}$, is unstable by blow-up. As we will need to solve a Cauchy problem in the proof, we assume that the restriction in Remark \ref{rem2.5},   namely $p\geq [s_{0}]+1=[{r\over 2}+{\rho \over 2}]+1$, holds below.
\begin{theorem}\label{theo4.8}
    Let $\epsilon > 0$ and $\Phi_{0}\in \mathcal{G}_{0}$. There exists initial data $U_{0}\in X$  with $\Vert U_{0}-\Phi_{0}\Vert_X <\epsilon$ for which the solution $U(t)$ of the Cauchy problem (\ref{sys1})-(\ref{sys3}) with initial data $U_{0}$ blows up in finite time.
\end{theorem}
\begin{proof}
    First, for $\lambda > 1$, consider $\lambda \Phi_{0}=\left(\lambda \phi_{0}, 0\right)$. Then
     \begin{eqnarray*}
     \mathcal{E}(\lambda \Phi_{0})&=&\lambda^{2} \mathcal{I}_{0}(\phi_{0}) -{\lambda^{p+1}\over {p+1}}\mathcal{Q}(\phi_{0}) \\
                            &=&\left( {\lambda^{2}\over 2}- {\lambda^{p+1}\over {p+1}} \right)\mathcal{Q}(\phi_{0})  \\
                                 &< &  \left( {1\over 2}- {1 \over {p+1}} \right)\mathcal{Q}(\phi_{0})=d(0).
     \end{eqnarray*}
     Also
    \begin{displaymath}
       2\mathcal{I}_{0}(\lambda \phi_{0})-\mathcal{Q}(\lambda \phi_{0}) =2\lambda^{2}\mathcal{I}_{0}(\phi_{0})-\lambda^{p+1}  \mathcal{Q}(\phi_{0})
                        = (\lambda^{2}-\lambda^{p+1})  \mathcal{Q}(\phi_{0}) <0.
    \end{displaymath}
    Next, as in  \cite{liu2}, we  define $v_{0}$ via Fourier transform:
    \begin{displaymath}
       \widehat{v_{0}}(\xi)={1\over {i\xi}}\widehat{\phi_{0}}(\xi) ~~ \mbox{for}~~ |\xi | \geq h, ~~~~~\mbox{and}~~~~~
       \widehat{v_{0}}(\xi)=0 ~~ \mbox{for}~~ |\xi | < h.
    \end{displaymath}
    Then $v_{0}\in L^{2}$. In fact, since $\phi_{0}\in H^{s_{0}}$, we have  $v_{0}\in H^{s_{0}+1}$   and thus $(v_{0})_{x}\in H^{s_{0}}$. For any $\epsilon>0$ we can choose $h$ sufficiently small such that $\Vert (v_{0})_{x}-\phi_{0}\Vert_{H^{s_{0}}}< \epsilon $. For $\lambda>1$ we let $U_{0} =\left(\lambda (v_{0})_{x}, 0\right)$. Since $\mathcal{E}$, $\mathcal{I}_{0}$, and $\mathcal{Q}$ are continuous on $H^{s_{0}}$ for $\lambda$ sufficiently close to 1, we get  $\Vert U_{0}-\Phi_{0}\Vert_X <\epsilon$,   $\mathcal{E}(U_{0}) < d (0)$ and $2\mathcal{I}_{0}(u_{0})-\mathcal{Q}(u_{0}) < 0$. But then $U_{0}$  satisfies the conditions of Theorem \ref{theo4.7}, and hence $U(t)$ will blow up in finite time.
\end{proof}
The next example illustrates the application of the above procedure to the Boussinesq equation. \\
{\it Example 1. (The   Boussinesq Equation)}
    If we set $L = I-\partial_{x}^{2}$ and $B = I$, we end up with  (\ref{gen-bouss}) and consequently with (\ref{gen-travel}) for which the solitary waves exist for $c^{2}<1$. Combining these with (\ref{d3}), after a straightforward calculation, we obtain the corresponding function $d(c)$ in the form
    \begin{displaymath}
        d (c) =d(0)(1-c^{2})^{{p+3}\over{2(p-1)}}
    \end{displaymath}
    where  $d(0)={1\over 2}\left({{p-1}\over {p+1}}\right)\left(\Vert \psi \Vert_{L^{2}}^{2}+\Vert \psi^{\prime} \Vert_{L^{2}}^{2}\right)$. Here the function $\psi$ satisfies $\psi^{\prime\prime}-\psi+| \psi |^{p-1}\psi=0$.
    Then we have
    \begin{displaymath}
      d^{\prime\prime} (c)=4 d(0) \frac{p+3}{(p-1)^{2}}(1-c^{2})^{\frac{7-3p}{2(p-1)}}\left(c^{2}-\frac{p-1}{4}\right).
    \end{displaymath}
    So,  when
    \begin{displaymath}
        \frac{p-1}{4}<  c^{2}<1   ~~~~\mbox{and}~~~~  1<p<5,
    \end{displaymath}
    $d(c)$ is convex and by Theorem \ref{theo4.4} the solitary wave solutions of (\ref{gen-bouss}) are orbitally stable.  This is exactly the same result which was obtained by Bona and Sachs \cite{bona2} for the stability of solitary wave solutions of (\ref{gen-bouss}).   On the other hand, Theorem \ref{theo4.4} is not applicable for small values of $c$ since the convexity assumption  is not valid. But  Theorem \ref{theo4.8} tells us that, for suitable initial data close to the standing wave, solutions of (\ref{gen-bouss}) blow up in finite time. For a more general case,  Liu \cite{liu1}  proved that the solitary waves of (\ref{gen-bouss}) are orbitally unstable in suitable function spaces if either
    \begin{displaymath}
        c^{2}\leq \frac{p-1}{4} ~~~\mbox{and}~~~ 1 < p< 5,
    \end{displaymath}
    or
    \begin{displaymath}
         c^{2} <1 ~~~\mbox{and}~~~ p\geq 5.
    \end{displaymath}
    As we have already mentioned,  Liu \cite{liu2} showed that  for $c = 0$,  the solitary waves are strongly unstable by blow-up, that is, certain solutions with initial data sufficiently close to $\phi_{0}$ blow up in finite time.   This result was  extended to the case of a small nonzero wave velocity  in \cite{liu3} and to the case of
    \begin{displaymath}
     0<  c^{2} < \frac{p-1}{2(p+1)}
    \end{displaymath}
    in \cite{liu-ohta}. For a recent discussion of these issues in the case of non-power nonlinearities, we refer the reader to \cite{howing}.

We now consider the double dispersion equation as a special case. \\
{\it Example 2. (The  Double Dispersion Equation)}
    When $L = (I-a_{1}\partial_{x}^{2})^{-1}(I-a_{2}\partial_{x}^{2})$ and $B = (I-a_{1}\partial_{x}^{2})^{-1}$ for two positive constants $a_{1}$ and $a_{2}$, (\ref{nonlocal}) reduces to (\ref{gen-doubly}). Since  $\rho=0$,  both regimes defined by (\ref{case-gen}) and (\ref{case-imp}) occur for the double dispersion equation. That is, solitary waves exist either for  $c^{2} <1 $ and $g(u)=-|u|^{p-1}u$ (i.e., the case $\rho \geq 0$ in Subsection 3.1  ) or for $c^{2} >1 $ and $g(u)=|u|^{p-1}u$ (i.e., the case $\rho \leq 0$ in Subsection 3.2  ).  Regarding the stability properties of solitary waves,  the comments made for the first regime are also valid for the double dispersion equation. We refer the reader to \cite{wang} for a strong instability result obtained in the first regime for that equation.

We conclude this section with the following remark regarding the case $\rho \leq 0$.
\begin{remark}\label{rem4.9}
    When $\rho \leq 0$, although $\phi_{c}$ is a minimizer for $\mathcal{J}_{c}$ (or a maximizer for $\mathcal{I}_{c}$) under a certain constraint, a variant of Lemma \ref{lem4.2} does not hold. In fact, at $\phi_{c}$ we have a saddle point of $\mathcal{E} (U) + c\mathcal{M} (U)$. This can be observed easily from $\mathcal{E}(U) +c\mathcal{M}(U)=\frac{1}{2}\left\Vert B^{-1/2}\left( w+cu\right) \right\Vert_{L^{2}}^{2}-\mathcal{J}_{c}(u) -\frac{1}{p+1}\mathcal{Q}(u)$. This is the main reason that the method used above for the case $\rho \geq 0$ will not work for the present case. In fact the case $\rho \leq 0$ corresponds to the "bad case" in \cite{stubbe}.  We now briefly indicate the results currently available in the literature for the the improved Boussinesq  equation which provides a prototype equation for the case  $\rho \leq 0$. Pego and Weinstein \cite{pego1} proved that solitary waves of (\ref{imp-boussinesq}) are linearly unstable in $H^{1}\times H^{2}$ if
    \begin{displaymath}
        1 < c^{2} <\frac{3(p-1)}{2(p+1)} ~~~\mbox{and}~~~ p>5.
    \end{displaymath}
     When $p=2$,  the linear instability of periodic traveling waves has  recently been shown  in \cite{pava2}.
\end{remark}
In the next section we study stability properties of the traveling waves for the case $L=I$.

\setcounter{equation}{0}
\section{An example: A regularized Klein-Gordon-type equation}
\noindent

 The previous section shows that orbital stability depends on the convexity of $d(c)$. In particular cases, for instance, in the case of the Boussinesq-type equations considered in the previous section, $d(c)$ can be computed explicitly using either the explicit form of the traveling wave solution $\phi_{c}$ or a Pohozaev-type identity, but both of these approaches will not work for the general case we deal with. In other words, we cannot get $d(c)$ explicitly  unless we  make further  assumptions on $L$ and/or $B$.  In this section we consider the particular case $L = I$ for which $\rho=0$ and $c_{1}=c_{2}=1$. Note that $s_{0}=s_{0}-{\rho \over 2}\equiv {r\over 2}  $. We will restrict our attention to the regime $c^{2}<1$  and $g(u)=-|u|^{p-1}u$. Taking $L=I$ allows us to compute $d(c)$  explicitly and hence determine the stability interval. Moreover, we are able to improve the instability result given in Theorem \ref{theo4.8} to get an almost complete characterization for stability of solitary waves in the first regime. When $L=I$, (\ref{nonlocal}) reduces to
\begin{equation}
    u_{tt}-u_{xx}=B(-|u|^{p-1}u)_{xx},    \label{klein}
\end{equation}%
 with the general pseudo-differential operator $B$ of order $-r$. Due to the smoothing effect of $B$, (\ref{klein}) can be considered as a regularized Klein-Gordon-type equation. Note that due to Theorem \ref{theo4.4} we need to take $r> 1$. We now give a full characterization of the orbital stability/instability of  traveling waves for (\ref{klein}) below.
  As we will need to solve a Cauchy problem in the proof of assertion (ii), we again assume that the restriction in Remark \ref{rem2.5},   namely $p\geq [s_{0}]+1=[{r\over 2}]+1$, holds below.
\begin{theorem}\label{theo5.1}
    Let $L=I$, $r>1$, $c^{2}<1$  and $g(u)=-|u|^{p-1}u$. Then
    \begin{enumerate}
    \item[(i)] For \ $c^{2}>\frac{p-1}{p+3}$, the traveling wave solutions of (\ref{sys1})-(\ref{sys3}) with velocity $c$ are orbitally stable.
    \item[(ii)] For $c^{2}<\frac{p-1}{p+3}$, the traveling wave solutions of (\ref{sys1})-(\ref{sys3}) with velocity $c$ are unstable by blow up; namely,  for any $\epsilon >0$ and $\Phi_{c}\in \mathcal{G}_{c}$ there exists initial data $U_{0} \in X$ with $\left\Vert U_{0} -\Phi_{c}\right\Vert _{X}<\epsilon $ for which the solution $U(t)$ of the Cauchy problem (\ref{sys1})-(\ref{sys3}) with initial data  $U_{0}$, blows up in finite time.
    \end{enumerate}
\end{theorem}
We first note from  (\ref{ic0}) that, for $L=I$
    \begin{displaymath}
            \mathcal{I}_{c}(u) =\frac{1}{2}(1-c^{2})\left\Vert B^{-{1/2}}u\right\Vert_{L^{2}}=(1-c^{2})\mathcal{I}_{0}(u).
    \end{displaymath}
So all the minimizers and hence $\phi_{c}$ traveling wave solutions are certain multiples of $\phi_{0}$, namely $\phi_{c}=(1-c^{2})^{\frac{1}{p-1}}\phi_{0}$.  From (\ref{d3}) we have $d(c)=d(0)(1-c^{2})^{\frac{p+1}{p-1}}$. Having disposed of this preliminary step, we can now easily prove the first assertion of Theorem \ref{theo5.1}. A straightforward computation gives
    \begin{displaymath}
        d^{\prime \prime }(c) = d(0)\frac{2( p+1)}{(p-1)^{2}}(1-c^{2})^{\frac{3-p}{p-1}}\left((p+3) c^{2}-p+1\right).
    \end{displaymath}%
Since $d\left( c\right) $ is strictly convex for $c^{2}>\frac{p-1}{p+3}$, it follows from Theorem \ref{theo4.4} that traveling waves are orbitally stable for $c^{2}>\frac{p-1}{p+3}$. This completes the proof of assertion (i) of Theorem \ref{theo5.1}.

The rest of this section will be devoted to the proof of assertion (ii) of Theorem \ref{theo5.1}. That is, we will prove that, when $c^{2}<\frac{p-1}{p+3}$, we can find initial data arbitrarily close to traveling wave solutions such that the solution of the corresponding Cauchy problem blows up in finite time.  Before proving the assertion, we need some preliminary definitions and results. Let us  first define a set $\Sigma _{-}(c)$ as follows.
\begin{displaymath}
    \Sigma _{-}(c)=\{(u,w)\in H^{s_{0}}\times H^{s_{0}-{\frac{\rho }{2}}}:
                    \quad \mathcal{E}(u,w)+c\mathcal{M}(u,w)<d(c),\quad 2\mathcal{I}_{c}(u)-\mathcal{Q}(u)<0\}.
\end{displaymath}%
The following lemma from \cite{erbay2}  shows that,  for  $L=I$ and $\ c^{2}<1$, the set $\Sigma _{-}(c)$ is  invariant under the flow generated by (\ref{sys1})-(\ref{sys3}).
\begin{lemma} \label{lem5.2} (Lemma 3.2 of \cite{erbay2})
    Suppose $(u_{0},w_{0})\in \Sigma _{-}(c)$, and let $(u(t),w(t))$ be the solution of the Cauchy problem (\ref{sys1})-(\ref{sys3}) with initial data $(u_{0},w_{0})$. Then $(u(t),w(t))\in \Sigma_{-}(c)$ for $0<t<T_{\max }$.
\end{lemma}
We also need the following lemma:
\begin{lemma}  \label{lem5.3}
    Suppose $2\mathcal{I}_{c}(u)-\mathcal{Q}(u)<0.$ Then $\ \frac{p+1}{p-1}d(c)<\mathcal{I}_{c}(u)$.
\end{lemma}
\begin{proof}
    Recall from (\ref{m1}) that $m_{1}\left( c\right) =\inf \left\{ \mathcal{I}_{c}(u):\mathcal{Q}(u)=1\right\}$. By homogeneity one gets
    \begin{displaymath}
       [m_{1}(c)]^{\frac{p+1}{2}}\leq \frac{[\mathcal{I}_{c}(u)]^{\frac{p+1}{2}}}{\mathcal{Q}(u)}
    \end{displaymath}
    whenever $u\not=0.$ If $2\mathcal{I}_{c}(u)-\mathcal{Q}(u)<0$ then
    \begin{displaymath}
    2[m_{1}(c)]^{\frac{p+1}{2}}\mathcal{I}_{c}(u)
        <  [m_{1}(c)]^{\frac{p+1}{2}}\mathcal{Q}(u)\leq [\mathcal{I}_{c}(u)]^{\frac{p+1}{2}}.
    \end{displaymath}
    Combining this with (\ref{d3}) yields
    \begin{displaymath}
    \frac{p+1}{p-1}d(c)=2^{\frac{2}{p-1}}[m_{1}(c)]^{\frac{p+1}{p-1}}<\mathcal{I}_{c}(u).
    \end{displaymath}
\end{proof}
We are now ready to prove the second assertion of Theorem \ref{theo5.1}:
\begin{proof}
Let $c^{2}<\frac{p-1}{p+3}$ and $\Phi_{c}=\left( \phi _{c},-c\phi _{c}\right) \in \mathcal{G}_{c}$.  We will follow the approach in Theorem
\ref{theo4.8} to construct initial data arbitrarily close to $\Phi_{c}$ such that the solution of the corresponding Cauchy problem blows up in finite time. For $\lambda >1$ consider $\lambda \Phi _{c}=\left( \lambda \phi_{c},-c\lambda \phi _{c}\right) $. Then,  just as in the proof of Theorem \ref{theo4.8},  we obtain
    \begin{eqnarray*}
    \mathcal{E}\left( \lambda\Phi_{c}\right) +c\mathcal{M}\left(\lambda \Phi_{c}\right)
            &=&\lambda^{2} \mathcal{I}_{c}(\phi_{c})-\frac{\lambda^{p+1}}{p+1}\mathcal{Q}(\phi_{c}) \\
            &=& \left( {\lambda^{2}\over 2}- {\lambda^{p+1}\over {p+1}} \right)\mathcal{Q}(\phi_{c})  \\
            &< &  \left( {1\over 2}- {1 \over {p+1}} \right)\mathcal{Q}(\phi_{c})=d(c),
    \end{eqnarray*}
    and
    \begin{displaymath}
    2\mathcal{I}_{c}(\lambda \phi_{c})-\mathcal{Q}(\lambda \phi_{c})
                    =2\lambda^{2}\mathcal{I}_{c}(\phi_{c})-\lambda^{p+1}\mathcal{Q}(\phi_{c})=(\lambda^{2}-\lambda^{p+1})\mathcal{Q}(\phi_{c}) <0.
    \end{displaymath}
    These two results show that  $\lambda \Phi _{c}=\left( \lambda \phi_{c}, -c\lambda \phi_{c}\right) \in \Sigma_{-}(c)$. Moreover,
    \begin{eqnarray*}
            -c\mathcal{M}\left( \lambda \Phi _{c}\right)
            &=& -c\lambda^{2}\mathcal{M}\left( \Phi_{c}\right) =c^{2}\lambda^{2}\left\Vert B^{-1/2}\phi_{c}\right\Vert_{L^{2}}^{2}
                        =\frac{2c^{2}\lambda^{2}}{1-c^{2}}\mathcal{I}_{c}\left( \phi_{c}\right) \\
            &>& \frac{2c^{2}}{1-c^{2}}\left(\frac{p+1}{p-1}\right)d(c)
    \end{eqnarray*}%
    where we have used (\ref{d3}). Next, as in the proof of Theorem \ref{theo4.8},  we choose some $v_{0}\in H^{s_{0}+1}$ such that $\left\Vert \left( v_{0}\right)_{x}-\phi_{c}\right\Vert _{H^{s_{0}}}<\epsilon $. For $\lambda >1$ we let  $U_{0}=\left( u_{0}, w_{0}\right) =\left( \lambda ( v_{0})_{x},-c\lambda (v_{0})_{x}\right) $. Since $\mathcal{E}$, $\mathcal{I}_{c}$, and $\mathcal{Q}$ are continuous on $H^{s_{0}}$ for $\lambda$ sufficiently close to 1, one gets: $\Vert U_{0}-\Phi_{c}\Vert_{X} <\epsilon$, $U_{0} \in \Sigma_{-}(c)$ and
    \begin{equation}
    -c\mathcal{M}\left( U_{0}\right)  > \frac{2c^{2}}{1-c^{2}}\left(\frac{p+1}{p-1}\right)d(c).  \label{estim}
    \end{equation}
    Let $U(t) =(u(t), w(t)) $ be the solution of the Cauchy problem (\ref{sys1})-(\ref{sys3}) with $L=I$. The rest of the proof is quite similar to the one of Theorem 3.5 of \cite{erbay2}. We then have $u=v_{x}$ with
    \begin{displaymath}
    v(., t)=\lambda v_{0}+\int_{0}^{t}  w(.,\tau)d\tau.
    \end{displaymath}
    With an easy computation this yields
    \begin{displaymath}
    \left\Vert B^{-1/2}v(t)\right\Vert_{L^{2}} \leq \lambda \left\Vert  B^{-1/2}v_{0}\right\Vert_{L^{2}}
            +\int_{0}^{t}  \left\Vert B^{-1/2}w(\tau)\right\Vert_{L^{2}}d\tau.
    \end{displaymath}
    This inequality tells us that $\left\Vert B^{-1/2}w(t)\right\Vert_{L^{2}}$,  equivalently $\left\Vert w(t)\right\Vert_{H^{r/2}}$, and thus $U(t)$  blows up in finite time whenever the functional $H(t)=\frac{1}{2}\left\Vert B^{-1/2}v(t)\right\Vert _{L^{2}}^{2}$ does so.  Therefore the proof is completed by showing that $H(t)$ blows up in finite time. Thanks to Levine's Lemma \cite{levine}. It says that if $H^{\prime }\left(t_{0}\right) >0$ for some $t_{0}>0$, and $\ HH^{\prime \prime }-\left( 1+\nu \right)\left( H^{\prime }\right) ^{2}\geq 0$ for some $\nu >0$  then  $H\left(t\right) $ will blow up in finite time. We proceed to show that
    \begin{eqnarray*}
    H^{\prime  }(t) &=& \left\langle B^{-1/2}v,B^{-1/2}v_{t}\right\rangle , \\
    H^{\prime \prime }(t)
        &=& \left\Vert B^{-1/2}v_{t}\right\Vert_{L^{2}}^{2}+\left\langle B^{-1/2}v,B^{-1/2}v_{tt}\right\rangle  \\
        &=& \left\Vert B^{-1/2}v_{t}\right\Vert_{L^{2}}^{2}+\int vB^{-1}v_{tt} dx  \\
        &=& \left\Vert B^{-1/2}v_{t}\right\Vert_{L^{2}}^{2}+\int v \left(B^{-1}v_{xx}-\left(|v_{x}|^{p-1}v_{x}\right)_{x}\right) dx  \\
        &=&  \left\Vert B^{-1/2}v_{t}\right\Vert_{L^{2}}^{2}-\int v_{x} \left(B^{-1}v_{x}-\left(|v_{x}|^{p-1}v_{x}\right)\right) dx \\
        &=&\left\Vert B^{-1/2}w\right\Vert_{L^{2}}^{2}-\left\Vert B^{-1/2}u\right\Vert _{L^{2}}^{2}+\mathcal{Q}(u) \\
        &=&\Vert B^{-1/2}(w+cu)\Vert^{2}_{L^{2}}-\left( 1+c^{2}\right) \left\Vert B^{-1/2}u\right\Vert _{L^{2}}^{2}-2c\mathcal{M}(u,w)+\mathcal{Q}(u) \\
        &=&\Vert B^{-1/2}(w+cu)\Vert^{2}_{L^{2}}-\frac{2\left( 1+c^{2}\right) }{1-c^{2}}\mathcal{I}_{c}\left( u\right) -2c\mathcal{M}(u,w)+\mathcal{Q}(u).
    \end{eqnarray*}%
    By (\ref{energy}), (\ref{momentum}), (\ref{ic0}) and (\ref{momentum-t}) we have
    \begin{displaymath}
    Q(u)=\frac{p+1}{2}\left\Vert B^{-1/2}(w+cu)\right\Vert^{2}_{L^{2}}+(p+1)\mathcal{I}_{c}(u)-(p+1)[\mathcal{E}(u,w)+c\mathcal{M}(u,w)].
    \end{displaymath}%
    Substituting this result into the above equation  we get
    \begin{eqnarray}
     H^{\prime \prime }\left( t\right)
    &=& \frac{p+3}{2}\left \Vert B^{-1/2}(w+cu)\right\Vert^{2}_{L^{2}} +\left( p+1-\frac{2( 1+c^{2}) }{1-c^{2}}\right) \mathcal{I}_{c}(u)   \nonumber \\
    && ~~~~~ -2c\mathcal{M}(u,w)-(p+1)[\mathcal{E}(u,w)+c\mathcal{M}(u,w)].  \label{hprime}
    \end{eqnarray}
    Note that the coefficient of $\mathcal{I}_{c}\left( u\right) $  is positive since $c^{2}<\frac{p-1}{p+3}$.  So the estimate of Lemma \ref{lem5.3}, i.e. $ \frac{p+1}{p-1}d(c)<\mathcal{I}_{c}(u)$,   can be employed above.  Furthermore, using the conservation laws we get
    \begin{displaymath}
    \mathcal{E}\left( U\right) +c\mathcal{M}(U)=\mathcal{E}\left(U_{0}\right) +c\mathcal{M}(U_{0})=d(c)-\delta <d(c)
    \end{displaymath}%
    for some $\delta >0$, and by (\ref{estim})
    \begin{displaymath}
    -c\mathcal{M}( U) =-c\mathcal{M}( U_{0}) > \frac{2c^{2}}{1-c^{2}}\left(\frac{p+1}{p-1}\right)d(c).
    \end{displaymath}
    Combining these with (\ref{hprime}) we obtain
    \begin{eqnarray*}
    H^{\prime \prime }(t)
    &>&\frac{p+3}{2}\left\Vert B^{-1/2}(w+cu)\right\Vert_{L^{2}}^{2}+\left( p+1-\frac{2\left( 1+c^{2}\right) }{1-c^{2}}\right) \left(\frac{p+1}{p-1}\right)d(c)  \\
    & & ~~~~~    +\frac{4c^{2}}{ 1-c^{2} }\left(\frac{p+1}{p-1}\right)d(c)-(p+1)d(c) +\left( p+1\right) \delta  \\
    &=&\frac{p+3}{2}\left\Vert B^{-1/2}(w+cu)\right\Vert_{L^{2}}^{2}+ \left( p+1-\frac{2\left(1+c^{2}\right) }{1-c^{2}} \right. \\
    & & ~~~~~  \left. +\frac{4c^{2}}{ 1-c^{2} }\right) \left(\frac{p+1}{p-1}\right)d(c)-(p+1)d( c) +( p+1) \delta  \\
    &=&\frac{p+3}{2}\left\Vert B^{-1/2}(w+cu)\right\Vert_{L^{2}}^{2}+\left( p+1\right) \delta .
    \end{eqnarray*}%
    So, $H^{\prime \prime }\left( t\right) >\left( p+1\right) \delta $  which in turn implies that $H^{\prime }\left( t_{0}\right) >0$ for some $ t_{0}>0$.  Thus, one of the two conditions of Levine's Lemma holds. What is left is to show that the second condition is also satisfied. Note that as $D_{x}$ commutes with $B^{-1/2}$ we have
    \begin{displaymath}
    \langle B^{-1/2}v,B^{-1/2}u\rangle
        =\int \left(B^{-1/2}v\right)\left( B^{-1/2}v\right) _{x}dx=\frac{1}{2}\int \frac{\partial }{\partial x}\left( B^{-1/2}v\right) ^{2}dx=0.
    \end{displaymath}%
    Since
    \begin{displaymath}
    \langle B^{-1/2}v,B^{-1/2}w\rangle =\langle B^{-1/2}v,B^{-1/2}\left(w+cu\right) \rangle
    \end{displaymath}
    we have
    \begin{displaymath}
    \left( H^{\prime }\left( t\right) \right) ^{2}
        =\left( \langle B^{-1/2}v,B^{-1/2}( w+cu) \rangle \right)^{2}\leq \Vert B^{-1/2}v\Vert _{L^2}^{2}\Vert B^{-1/2}( w+cu) \Vert _{L^2}^{2}.
    \end{displaymath}
    Finally, with \ $1+\nu =\frac{p+3}{4}$ \ we have
   \begin{displaymath}
    H\left( t\right) H^{\prime \prime }(t)-\frac{p+3}{4}\left( H^{\prime }\left(t\right) \right)^{2}
        \geq \frac{p+1}{2}\Vert B^{-1/2}v\Vert _{L^2}^{2} \delta =(p+1) H(t)\delta \geq 0
    \end{displaymath}
    This completes the proof of assertion (ii) of Theorem \ref{theo5.1}
\end{proof}
\vspace*{20pt}

\noindent
{\bf Acknowledgement}: This work has been supported by the Scientific and Technological Research Council of Turkey (TUBITAK) under the project TBAG-110R002.


\begin{thebibliography}{99}

\bibitem{babaoglu} C. Babaoglu, H.A. Erbay, A. Erkip,
                    Global existence and blow-up of solutions for a general class of doubly dispersive nonlocal nonlinear wave equations, Nonlinear Anal. 77 (2013) 82-93.  

\bibitem{erbay2} H.A. Erbay, S. Erbay, A. Erkip,
                    Thresholds for global existence and blow-up in a general class of doubly dispersive nonlocal nonlinear wave equations, Nonlinear Anal. 95 (2014) 313-322. 

\bibitem{bous} J. Boussinesq,
                    Theorie des ondes et des remous qui se propagent le long d'un canal rectangulaire horizontal, en communiquant au liquide contenu dans ce canal des vitesses sensiblement pareilles de la surface au fond, Journal de Mathematiques Pures et Appliquees 17 (1872) 55-108. 

\bibitem{ost}  L.A. Ostrovskii, A.M. Sutin,
                    Nonlinear elastic waves in rods, PMM J. Appl. Math. Mech. 41 (1977) 543-549. 

\bibitem{samsonov} Samsonov A M 2001
                   {\it Strain Solitons in Solids and How to Constuct Them} (Boca Raton:  Chapman and Hall)

\bibitem{porubov}  Porubov A V 2003
                    {\it Amplification of Nonlinear Strain Waves in Solids} (Singapore: World Scientific)

\bibitem{berezovski} A. Berezovski, J. Engelbrecht, A. Salupere, K. Tamm, T. Peets, M. Berezovski,
                    Dispersive waves in microstructured solids, Int. J. Solids Struct. 50 (2013) 1981-1990.

\bibitem{narendar} S. Gopalakrishnan, S. Narendar,
                    Wave Propagation in Nanostructures: Nonlocal Continuum Mechanics Formulations, Springer,  Switzerland, 2013.

\bibitem{duruk1} N. Duruk, H.A. Erbay, A. Erkip,
                    Global existence and blow-up for a class of nonlocal nonlinear Cauchy problems arising in elasticity, Nonlinearity 23 (2010) 107-118. 

\bibitem{erbay1} H.A. Erbay, S. Erbay, A. Erkip,
                    The Cauchy problem for a class of two-dimensional nonlocal nonlinear wave equations governing anti-plane shear motions in elastic materials, Nonlinearity 24 (2011) 1347-1359. 

\bibitem{duruk2} N. Duruk, H.A. Erbay, A. Erkip,
                    Blow-up and global existence for a general class of nonlocal nonlinear coupled wave equations, J. Differential Equations 250 (2011) 1448-1459. 

\bibitem{benjamin} T.B. Benjamin,
                    The stability of solitary waves, Proc. R. Soc. Lond. Ser. A  328 (1972) 153-183. 

\bibitem{bona1} J.L. Bona,
                    On the stability theory of solitary waves, Proc. R. Soc. Lond. Ser. A 344 (1975) 363-374. 

\bibitem{pava1} J.A. Pava,
                    Nonlinear Dispersive Equations: Existence and Stability of Solitary and Periodic Travelling Wave Solutions, AMS Mathematical Surveys and Monographs 156, American Mathematical Society, Providence, 2009.  

\bibitem{bona2} J.L. Bona, R. Sachs,
                    Global existence of smooth solutions and stability of solitary waves for a generalized Boussinesq equation, Comm. Math. Phys.
                    118 (1988) 15-29. 

\bibitem{liu1} Y. Liu,
                    Instability of solitary waves for generalized Boussinesq equations, J. Dynam. Differential Equations  5 (1993) 537-558. 

\bibitem{pego1} R.L. Pego, M.I. Weinstein,
                    Eigenvalues, and instabilities of solitary waves, Philos. Trans. R. Soc. Lond. Ser. A 340 (1992) 47-94. 

\bibitem{esfahani} A. Esfahani, S. Levandosky,
                    Stability of solitary waves for the generalized higher-order Boussinesq equation, J. Dynam. Differential Equations 24 (2012) 391-425. 

\bibitem{bona3} J.L. Bona, P.E. Souganidis, W.A. Strauss,
                    Stability and instability of solitary waves of Kortewg-de Vries type, Proc. R. Soc. Lond. Ser. A 411 (1987) 395-412. 

\bibitem{albert0} J.P. Albert, J.L. Bona, D.B. Henry,
                    Sufficient conditions for stability of solitary-wave solutions of model equations for long waves, Physica D 24 (1987) 343-366. 

\bibitem{souganidis} P.E. Souganidis, W.A. Strauss,
                    Instability of a class of dispersive solitary waves, Proc. Roy. Soc. Edinburgh Sect. A 114 (1990) 195-212. 

\bibitem{albert1} J.P. Albert, J.L. Bona, J.C. Saut,
                    Model equations for waves in stratified fluids, Proc. R. Soc. Lond. Ser. A 453 (1997) 1233-1260. 

\bibitem{albert2} J.P. Albert,
                    Concentration compactness and the stability of solitary-wave solutions to nonlocal equations, Contemporary Mathematics 221 (1999)  1-29. 

\bibitem{albert3} J.P. Albert, F. Linares,
                    Stability and symmetry of solitary-wave solutions to systems modeling interactions of long waves,  Journal de Mathématiques Pures et Appliquées 79 (2000) 195-226. 

\bibitem{zeng} L. Zeng,
                    Existence and stability of solitary-wave solutions of equations of Benjamin-Bona-Mahony type, J. Differential Equations  188 (2003) 1-32. 

\bibitem{mats} M. Ehrnstr\"om, M.D. Groves, E. Wahlen,
                    On the existence and stability of solitary-wave solutions to a class of evolution equations of Whitham type,  Nonlinearity 25 (2012)  2903-2936. 

\bibitem{stubbe} J. Stubbe,
                    Existence and stability of solitary waves of Boussinesq-type equations, Portugaliae Mathematica 46 (1989) 501-516. 

\bibitem{bronski} J. Bronski, M.A. Johnson, T. Kapitula,
                    An instability index theory for quadratic pencils and applications, Commun. Math. Phys. 327 (2014) 521–550. 

\bibitem{hakkaev} S. Hakkaev, M. Stanislavova, A. Stefanov,
                    Linear stability analysis for periodic traveling waves of the Boussinesq equation and the Klein-Gordon-Zakharov system, Proc. Roy. Soc. Edinburgh A 144 (2014) 455-489. 

\bibitem{stanislavova}  M. Stanislavova, A. Stefanov,
                    Linear stability analysis for travelling waves of second order in time PDE's, Nonlinearity 25 (2012) 2625–2654. 

\bibitem{lions1} P.L. Lions,
                    The concentration-compactness principle in the caculus of variation: The locally compact case part 1, Ann. Inst. H. Poincare Anal. Non Lineaire  1 (1984) 109-145. 

\bibitem{lions2} P.L. Lions,
                    The concentration-compactness principle in the caculus of variation: The locally compact case part 2, Ann. Inst. H. Poincare Anal. Non Lineaire 1 (1984)223-283. 

\bibitem{levandosky} S. Levandosky,
                    Stability and instability of fourth-order solitary waves, J. Dynam. Differential Equations 10 (1998),  151-188. 

\bibitem{liu2} Y. Liu,
                    Instability and blow-up of solutions to a generalized Boussinesq equation, SIAM J. Math. Anal. 26 (1995) 1527-1545. 

\bibitem{liu3} Y. Liu,
                    Strong instability of solitary-wave solutions of a generalized Boussinesq equation,  J. Differential Equations 164 (2000) 223-239. 

 \bibitem{liu-ohta} Y. Liu, M. Ohta, G. Todorova,
                     Strong instability of solitary waves for nonlinear Klein– Gordon equations and generalized Boussinesq equations, Ann. Inst. H. Poincare Anal. Non Lineaire  24 (2007) 539-548. 

\bibitem{coifman} R. Coifman, Y. Meyer,
                    Au-del\`a des op\'erateurs pseudo-diff\'erentiels,  Ast\'erisque 57, Societe Mathematique Francaise, France, 1978. 

\bibitem{howing} J. H\"owing,
                    Stability of large- and small-amplitude solitary waves in the generalized Korteweg-de Vries and Euler-Korteweg/Boussinesq equations,  J. Differential Equations 251 (2011) 2515-2533. 

\bibitem{wang} Y. Wang, C. Mu, J. Deng,
                    Strong instability of solitary-wave solutions for a nonlinear Boussinesq equation, Nonlinear Anal. 69 (2008) 1599-1614. 

\bibitem{pava2} J.A. Pava, C. Banquet, J.D. Silva, F. Oliveira,
                    The Regularized Boussinesq equation: Instability of periodic traveling waves, J. Differential Equations 254 (2013) 3994-4023.  

\bibitem{levine} H.A. Levine,
                    Instability and nonexistence of global solutions to nonlinear wave equations of the form $Pu_{tt}=-Au+f(u)$, Transactions of American Mathematical Society 192 (1974) 1-21. 

\end{thebibliography}
\end{document}